\documentclass[UTF-8,reqno]{amsart}
\usepackage{enumerate}
\usepackage{mhequ}
\usepackage{amssymb,url,color, booktabs,nccmath,amsmath}
\usepackage[left=3.2cm,right=3.2cm,top=4cm,bottom=4cm]{geometry}
\usepackage{mathrsfs}
\usepackage{enumitem,dsfont}
\usepackage{color}
\usepackage[colorlinks=true]{hyperref}
\hypersetup{
	linkcolor=blue,          % color of internal links
	citecolor=red,        % color of links to bibliography
	filecolor=blue,      % color of file links
	urlcolor=cyan
}

\definecolor{darkergreen}{rgb}{0.0, 0.5, 0.0}

\numberwithin{equation}{section}
\def\theequation{\arabic{section}.\arabic{equation}}
\newcommand{\be}{\begin{eqnarray}}
	\newcommand{\ee}{\end{eqnarray}}
\newcommand{\ce}{\begin{eqnarray*}}
	\newcommand{\de}{\end{eqnarray*}}
\newtheorem{theorem}{Theorem}[section]
\newtheorem{lemma}[theorem]{Lemma}
\newtheorem{proposition}[theorem]{Proposition}
\newtheorem{Examples}[theorem]{Example}
\newtheorem{corollary}[theorem]{Corollary}

\newtheorem{definition}[theorem]{Definition}
\theoremstyle{definition}
\newtheorem{remark}[theorem]{Remark}

\def\${|\!|\!|}

\DeclareMathOperator{\supp}{supp}

\def\eps{\varepsilon}

\def\<{{\langle}}
\def\>{{\rangle}}
\def\({{\Big(}}
\def\){{\Big)}}

\def\bx{{\mathbf{x}}}
\def\tr{\mathrm {tr}}

\def\dif{{\mathord{{\rm d}}}}

\def\min{{\mathord{{\rm min}}}}

\def\={&\!\!=\!\!&}

\def\cT{{\mathcal T}}

\def\mN{{\mathbb N}}

\def\mR{{\mathbb R}}

\def\mT{{\mathbb T}}

\def\bP{{\mathbf P}}

\def\1{{\mathbf{1}}}

\def\E{\mathbf E}

\def\geq{\geqslant}
\def\leq{\leqslant}

\def\div{\mathord{{\rm div}}}

\def\eps{\varepsilon}

\def\<{{\langle}}
\def\>{{\rangle}}
\def\({{\Big(}}
\def\){{\Big)}}

\def\bx{{\mathbf{x}}}
\def\tr{\mathrm {Tr}}

\def\dif{{\mathord{{\rm d}}}}

\def\min{{\mathord{{\rm min}}}}

\def\={&\!\!=\!\!&}
\def\bt{\begin{theorem}}
	\def\et{\end{theorem}}
\def\bl{\begin{lemma}}
	\def\el{\end{lemma}}
\def\br{\begin{remark}}
	\def\er{\end{remark}}
\def\bx{\begin{Examples}}
	\def\ex{\end{Examples}}
\def\bd{\begin{definition}}
	\def\ed{\end{definition}}
\def\bp{\begin{proposition}}
	\def\ep{\end{proposition}}
\def\bc{\begin{corollary}}
	\def\ec{\end{corollary}}

\def\geq{\geqslant}
\def\leq{\leqslant}

\def\div{\mathord{{\rm div}}}

\def\Id{\textrm{Id}}

\def\bP{{\mathbf P}}

 \def\R{\mathbb R}
 \def\R{\mathbb R}    
\def\N{\mathbb N}  
   
\def\<{\langle} \def\>{\rangle}

\makeatletter

\newcommand{\Rmnum}[1]{\expandafter\@slowromancap\romannumeral #1@}
\makeatother

\allowdisplaybreaks

\begin{document}
	\fontsize{10.0pt}{\baselineskip}\selectfont
	
	\title{Stationary solutions to stochastic 3D Euler equations in H\"{o}lder space}
	
	\author{Lin L\"{u}}  
	\address[L. L\"{u}]{Department of Mathematics, Beijing Institute of Technology, Beijing 100081, China}
	\email{3120235976@bit.edu.cn}
	
	\author{Rongchan Zhu}
	\address[R. Zhu]{Department of Mathematics, Beijing Institute of Technology, Beijing 100081, China}
	\email{zhurongchan@126.com}
	
	\thanks{ Research supported by National Key R\&D Program of China (No. 2022YFA1006300) and the NSFC (No.   12271030). The financial support by the DFG through the CRC 1283 ”Taming uncertainty and profiting from randomness and low regularity in analysis, stochastics and their applications” is greatly acknowledged.}
	
	\begin{abstract}
		We establish the existence of infinitely many global and stationary solutions in $C(\R,C^{\vartheta})$ space for some $\vartheta>0$ to the three-dimensional Euler equations driven by an additive stochastic forcing. The result is based on a new stochastic version of the convex integration method, incorporating the stochastic convex integration method developed in \cite{HZZ22b} and pathwise estimates to derive uniform moment estimates independent of time.
	\end{abstract}
	
	
	\maketitle
	\tableofcontents
	\section{Introduction}
	In this paper, we consider the stochastic Euler equations governing the time evolution of the velocity $u$ of an inviscid fluid on the three-dimensional torus $\mathbb{T}^3=\mathbb{R}^3/(2\pi \mathbb{Z})^3$. The system reads as
	\begin{equation}\label{eul1}
		\aligned
		\dif u+\div(u\otimes u)\,\dif t+\nabla P\,\dif t&=\dif B,
		\\
		\div u&=0,
		\endaligned
	\end{equation}
	where $P$ stands for the corresponding pressure. The right hand side represents a random external force acting on the fluid, which is given by a $GG^*$-Wiener process $B$ defined on some probability space $(\Omega,\mathcal{F},(\mathcal{F}_{t})_{t\in\mR},\mathbf{P})$ with $GG^*$ being trace class operator. 	The notion of solution to \eqref{eul1} we use throughout this paper is that of  analytically weak solutions, here recalled.
	\begin{definition}\label{d:sol}
		We say that $((\Omega,\mathcal{F},(\mathcal{F}_{t})_{t\in\mR},\mathbf{P}),u,B)$ is an analytically weak solution to the stochastic Euler system \eqref{eul1} provided
		\begin{enumerate}
			\item $(\Omega,\mathcal{F},(\mathcal{F}_{t})_{t\in\mR},\mathbf{P})$ is a stochastic basis with a complete right-continuous filtration;
			\item $B$ is an $\R^{3}$-valued, mean zero and divergence-free, two-sided $GG^*$-Wiener process with respect to the  filtration $(\mathcal{F}_{t})_{t\in\mR}$;
			\item the velocity $u \in  C(\mR \times \mathbb{T}^{3})$ $\mathbf{P}$-a.s. and is $(\mathcal{F}_{t})_{t\in\mR}$-adapted;
			\item for every $-\infty<s\leq t<\infty$ it holds  $\mathbf{P}$-a.s.
			$$
			\begin{aligned}
				&\langle u(t),\psi \rangle  =\langle u(s),\psi \rangle + \int_{s}^{t} \langle   u,u \cdot \nabla \psi \rangle  \dif r +\langle B(t)-B(s),\psi\rangle
			\end{aligned}
			$$
			for all $\psi\in C^{\infty}(\mathbb{T}^{3},\mR^3)$, $\div \psi=0$.
			
		\end{enumerate}
	\end{definition}
	
	Stochastic Euler equations have been extensively studied in both two and three dimensions in previous works, including \cite{BFM16, BP01, CFH19, GV!4}, where local-in-time solutions for \eqref{eul1} have been constructed. Our objective is to investigate the existence of global and stationary solutions in the sense of Definition~\ref{d:sol} with H\"{o}lder continuity to the above system \eqref{eul1}. To achieve this goal, we employ the method of convex integration, a technique that has gained prominence in fluid dynamics over the past decades. This method was initially introduced following \cite{DLS10, DLS12, DelSze13}  by De Lellis and Sz\'ekelyhidi Jr. In these works, the convex integration method was used to prove the non-uniqueness of weak solutions with H\"{o}lder continuity to Euler equations. Subsequently, this method has led to several breakthroughs in fluid dynamics, including proof of Onsager’s conjecture for the incompressible Euler equations \cite{Ise18,BDLSV19}. We refer to the excellent review article by Buckmaster and Vicol \cite{BV19} for a gentle introduction and further references.
	
	In recent years, convex integration has already been applied in a stochastic setting, namely, to the isentropic Euler system by Breit, Feireisl and Hofmanov\'a \cite{BFH20} and to the full Euler system by Chiodaroli, Feireisl and Flandoli \cite{CFF19} with linear multiplicative noise. Furthermore, the ill/well-posedness of dissipative martingale solutions to stochastic 3D Euler equations was studied by Hofmanov\'a, Zhu and Zhu \cite{HZZ22a}, along with the existence and non-uniqueness of strong Markov solutions. Later, H\"{o}lder continuous, global-in-time probabilistically strong solutions to 3D Euler equations perturbed by Stratonovich transport noise was constructed in \cite{HLP22}. In the most recent work \cite{HZZ22b}, the existence and non-uniqueness of stationary solutions in $H^\vartheta$ to the stochastic 3D Euler equations was established, using a new stochastic convex integration method.
	
	However, the regularity of the global and stationary solutions to the stochastic 3D Euler equations with additive noise, as constructed in the previous work, cannot achieve H\"{o}lder continuity. Inspired by these previous works and the progress in the deterministic case, it is natural to investigate the existence and non-uniqueness of global and stationary Hölder continuous solutions to the Euler equations in the stochastic setting \eqref{eul1}. Another motivation for our research is linked to the Onsager's conjecture. It is constructed in \cite{Ise18, BDLSV19} that solutions in H\"{o}lder space $C^{\frac13-}$ to the deterministic Euler equations do not conserve energy. It is natural to ask whether the global and nonunique solutions in the stochastic setting can stay $C^{\frac13-}$ almost surely.
	
	To summarize the above discussion, we investigate the validity of the following claims.
	\begin{enumerate}
		\item [(i)] Existence and non-uniqueness of global $\vartheta-$H\"{o}lder continuous solutions to the stochastic Euler equations \eqref{eul1}, for some $\vartheta>0$.
		\item [(ii)]  Existence and non-uniqueness of stationary solutions to the stochastic Euler equations \eqref{eul1}.
	\end{enumerate}

	Our objective is to address the aforementioned issues (i) and (ii) within the context of the stochastic Euler equations on $\mathbb{T}^3$ driven by additive stochastic noise. We employ a stochastic version of the convex integration method as the primary tool in this study. Unlike previous works applying convex integration for the Navier--Stokes and Euler equations in the stochastic setting \cite{HZZ22a, HZZ21, HZZ21markov, HZZ19},  our approach is inspired by \cite{CDZ22} and \cite{HZZ22b}. In these works, the limitations stemming from stopping times, which were previously used to control the noise terms in the iteration, are overcome. To be more specific, we have moved away from working with stopping times. Instead, we incorporate expectations into the inductive iteration and provide uniform moment estimates. However, the construction of H\"{o}lder continuous Euler flows in the deterministic setting, as demonstrated in \cite{BV19}, requires to introduce transport equations, and pathwise positive lower bounds for the solutions are used to improve regularity. As a result, we also introduce some pathwise estimates during iteration for convex integration and use the cutoff technique to control the growth of noise pathwisely. Compared to only using stochastic convex integration, the combination of pathwise estimates with it can provide better H\"{o}lder regularity of the solutions (see Remark~\ref{remark} below). This approach also allows us not to control all the moment bounds during the iteration. We hope this cutoff method can be applied to enhance the regularity of the solutions in future work.
		\subsection{Main results}\label{main}
	Our first main result reads as follows and is proved in Theorem~\ref{v1}, tackling the problem (i) above.
	\bt\label{Thm1}
	Suppose that $\tr((-\Delta)^{\frac32+\sigma}GG^*)<\infty$ for some $\sigma>0$, then there exists an analytically weak solution to \eqref{eul1} in the sense of Definition~\ref{d:sol} with $(\mathcal{F}_{t})_{t\in\mR}$ being the normal filtration generated by the Wiener process $B$, which belongs to $C(\R,C^\vartheta)$ for some $\vartheta\in \left( 0,\min{\{ \frac{\sigma}{120{\cdot}7^5},\frac{1}{3{\cdot}7^5}\}}\right) $.There are infinitely many such solutions $u$. 
	\et
	\begin{remark}
	In \cite{HZZ22b}, the authors have constructed solutions to \eqref{eul1} with continuous paths in a suitable Sobolev space, namely $C(\mR, H^{\vartheta})$ $\mathbf{P}$-almost surely.  In contrast, we are able to construct solutions achieving the  H\"{o}lder continuity. Additionally, the stochastic convex integration developed in \cite{HZZ22b} requires estimating all the moments of the solutions, while our analysis only necessitates specifying an iteration procedure for a finite number of moments, along with additional pathwise bounds.
	
	\end{remark}
	\begin{remark}\label{remark}
		We can also establish Theorem~\ref{Thm1} in a separate article (see \cite{LZ24}), which sorely employs the stochastic convex integration method to provide uniform moment bounds. In other words, we do not incorporate pathwise estimates (see \eqref{vqaa}, \eqref{vqcc}, \eqref{vqdd} below) during the iteration for convex integration. In that paper,  the construction of building blocks of the convex integration method follows from \cite{DelSze13}, and we do not require pathwise lower bounds of the solutions to the transport equation \eqref{eq:te} below. Therefore, the stochastic convolution $z$ is not cut off. However, the value of $\vartheta$ in this case decreases due to the arbitrarily higher moments required to be estimated, and we may take $\vartheta<\min{\{\frac{\sigma}{73{\cdot}62^3},\frac{1}{7{\cdot}62^3}\}}$.
       
	\end{remark}
	
	Our second result establishes the existence and non-uniqueness of stationary solutions in the H\"{o}lder space. Here, stationarity is interpreted in terms of the shift invariance of laws governing solutions on the space of trajectories, as discussed in \cite{BFHM19, BFH20e, FFH21, HZZ22b}. To be more specific, we define the joint trajectory space for the solution and the driving Wiener process as follows:
	\begin{align*}
		\cT = C(\mR,C^{\kappa})\times C(\mR,C^{\kappa})
	\end{align*}
	for some $\kappa>0$. And let $S_t$, $t\in\mR$, be  shifts on trajectories given by
	\begin{align*}
		S_t(u,B)(\cdot)=(u(\cdot+t),B(\cdot+t)-B(t)),\quad t\in\mR,\quad (u,B)\in\cT.
	\end{align*}
	We observe that the shift in the second component operates differently to ensure that for a Wiener process $B$, the shift $S_{t}B$ remains a Wiener process.
	
	\begin{definition}\label{d:1.1}
		We say that  $((\Omega,\mathcal{F},(\mathcal{F}_{t})_{t\in\mR},\mathbf{P}),u,B)$ is a stationary solution to the stochastic Euler equations \eqref{eul1} provided it satisfies  \eqref{eul1} in the sense of Definition~\ref{d:sol}  and its law is shift invariant, that is,
		$$\mathcal{L}[S_{t}(u,B)]=\mathcal{L}[u,B]\qquad\text{ for all }\quad t\in\mR.$$
	\end{definition}
	
	With these definitions at hand, our second main result then reads as follows and is proved in Theorem~\ref{th:s1}, tackling the problem (ii) above.

	\bt\label{th:main}
	Suppose that $\tr((-\Delta)^{\frac32+\sigma}GG^*)<\infty$ for some $\sigma>0$, then there exist infinitely many stationary solutions
	to the stochastic Euler equations \eqref{eul1}. Moreover, the solutions belong to $C(\mR,C^{\vartheta})$ a.s. for some $\vartheta>0$, satisfying
	$$\sup_{t\in\R}\mathbf{E}\left[\sup_{t\leq s \leq t+1}\|u(s)\|_{C^{\vartheta}}\right]<\infty. $$
	\et
	\subsection*{Organization of the paper}Section~\ref{notation} compiles the fundamental notations used throughout the paper. Section~\ref{ci2}, \ref{convex}, \ref{s:it1} constitute the core of our proofs, where stochastic convex integration is developed and employed to construct entire analytically weak solutions. In Section~\ref{ci2}, we present the main iteration Proposition~\ref{p:iteration1} to prove our main Theorem~\ref{v1}. Section~\ref{convex} and Section~\ref{s:it1} are devoted to the construction of stochastic convex integration with pathwise estimates, namely, proof of iteration Proposition~\ref{p:iteration1}. 
	Then the results from Section~\ref{ci2} are applied in Section~\ref{s:4}, coupled with a Krylov--Bogoliubov's argument, to establish the existence of non-unique stationary solutions to the stochastic Euler equations. 
	In Appendix~\ref{Bw2}, we recall the construction and property of Beltrami waves from \cite{BV19} . In Appendix \ref{ap:A'}, we estimate amplitude functions used in the convex integration construction. Finally, the estimates for transport equations are completed in Appendix~\ref{s:B}.

	\section{Preliminaries}\label{notation}
	In the sequel, we use the notation $a\lesssim b$ if there exists a constant $c>0$ such that $a\leq cb$. 
	\subsection{Function spaces}\label{s:2.1}
	Given a Banach space $X$ with a norm $\|\cdot\|_X$ and $t\in\mR$, we denote $C_tX:=C([t,t+1],X)$ as the space of continuous functions from $[t,t+1]$ to $X$, equipped with the supremum norm 
	$$\|f\|_{C_tX}:=\sup_{s\in[t,t+1]}\|f(s)\|_{X}.$$  
	For $\alpha\in(0,1)$, we use $C^\alpha_tX$ to denote the space of $\alpha$-H\"{o}lder continuous functions from $[t,t+1]$ to $X$, endowed with the norm $$\|f\|_{C^\alpha_tX}:=\sup_{s,r\in [t,t+1],s\neq r}\frac{\|f(s)-f(r)\|_X}{|r-s|^\alpha}+\|f\|_{C_tX}.$$ 
	We also write $C^0 :=C(\mathbb{T}^3,\mR^3)$ equipped with the supremum norm $\|f\|_{C^0}:=\sup_{x\in \mathbb{T}^3}|f(x)|$.
	For $\alpha \in (0,1)$, we write $C^\alpha:=C^\alpha(\mathbb{T}^3,\mR^3)$ as the space of $\alpha$-H\"{o}lder continuous functions from $\mathbb{T}^3$ to $\mR^3$ endowed with the norm 
	$$\|f\|_{C^\alpha}:=\sup_{x,y\in \mathbb{T}^3,x\neq y}\frac{|f(x)-f(y)|}{|x-y|^\alpha}+\sup_{x\in \mathbb{T}^3}|f(x)|.$$
   We denote $L^p$ as the set of  standard $L^p$-integrable functions from $\mathbb{T}^3$ to $\mathbb{R}^3$. For $s>0$, $p>1$ the Sobolev space $W^{s,p}:=\{f\in L^p; \|f\|_{W^{s,p}}:= \|(I-\Delta)^{{s}/{2}}f\|_{L^p}<\infty\}$. We set $L^{2}_{\sigma}:=\{f\in L^2; \int_{\mathbb{T}^{3}} f\,\dif x=0,\div f=0\}$. For $s>0$, we also denote $H^s:=W^{s,2}\cap L^2_\sigma$. 
	For $t\in\mathbb{R}$ and a domain $D\subset\R$ and $N\in\N_{0}:=\mN\cup \{0\}$, we denote the space of $C^{N}$-functions from $[t,t+1]\times\mathbb{T}^{3}$ and $D\times\mathbb{T}^{3}$ to $\mR^3$, respectively, by  $C^{N}_{t,x}$ and $C^{N}_{D,x}$. The spaces are equipped with the norms
	$$
	\|f\|_{C^N_{t,x}}=\sum_{\substack{0\leq n+|\alpha|\leq N\\ n\in\N_{0},\alpha\in\N^{3}_{0} }}\|\partial_t^n D^\alpha f\|_{L^\infty_{[t,t+1]} L^\infty},\qquad \|f\|_{C^N_{D,x}}=\sum_{\substack{0\leq n+|\alpha|\leq N\\ n\in\N_{0},\alpha\in\N^{3}_{0} }}\sup_{t\in D}\|\partial_t^n D^\alpha f\|_{ L^\infty}.
	$$
	We also use $\mathring{\otimes}$ to denote the trace-free part of the tensor product. For a tensor $T$, we denote its traceless part by $\mathring{T}:=T-\frac13\tr(T)\rm{Id}$.
	
	We recall  \cite[Definition 4.2]{DelSze13} the inverse divergence operator $\mathcal{R}$ which acts on vector fields $v$ with $\int_{\mathbb{T}^3}v\dif x=0$ as
	\begin{equation*}
		(\mathcal{R}v)^{kl}=(\partial_k\Delta^{-1}v^l+\partial_l\Delta^{-1}v^k)-\frac{1}{2}(\delta_{kl}+\partial_k\partial_l\Delta^{-1})\div\Delta^{-1}v,
	\end{equation*}
	for $k,l\in\{1,2,3\}$. The above inverse divergence operator has the property that  $\mathcal{R}v(x)$ is a symmetric trace-free matrix for each $x\in\mathbb{T}^3$, and $\mathcal{R}$ is a right inverse of the div operator, i.e. $\div(\mathcal{R} v)=v$. By \cite[Theorem B.3]{CL22} we have for $1\leq p\leq \infty$
	\begin{align}\label{eR}
		\|\mathcal{R}f\|_{L^p(\mathbb{T}^3)}\lesssim \|f\|_{L^p(\mathbb{T}^3)}.
	\end{align}
	
	\subsection{Probabilistic elements}\label{s:2.2}
	For a given probability measure $\mathbf{P}$,  we use $\mathbf{E}$ to denote the expectation under $\mathbf{P}$.
	Concerning the driving noise, we assume that $B$ is $\R^{3}$-valued two-sided  $GG^*$-Wiener process with mean zero and divergence-free. This process is defined on some probability space $(\Omega, \mathcal{F}, \mathbf{P})$, where $G$ is a Hilbert--Schmidt operator from $U$ to $L_{\sigma}^2$ for some Hilbert space $U$.
	
	Given a Banach space $X=C(\mathbb{T}^3,\mR^3)$ or $X=C^{\kappa}(\mathbb{T}^3,\mR^3)$ for some $\kappa>0$. For $p\in[1,\infty)$ and $\delta \in (0,\frac12)$, we denote
\begin{align*}
	\$u\$^p_{X,p}:=\sup_{t\in \mathbb{R}}\mathbf{E}\left[\sup_{s\in [t,t+1]}\|u(s)\|^p_X\right], \quad  \$u\$^p_{C_t^{\frac{1}{2}-\delta}X,p}:=\sup_{t\in \mR} \mathbf{E}\left[ \|u\|^p_{C_t^{\frac{1}{2}-\delta}X}\right].
\end{align*}
	The above norms denote function spaces of random variables on $\Omega$ taking values in $C(\mR,X)$ and $C^{\frac{1}{2}-\delta}(\mR,X)$, respectively,  with bounds in $L^{p}(\Omega;{C(I,X))}$ and $L^{p}(\Omega;{C^{\frac{1}{2}-\delta}(I,X)})$ for  
	bounded interval $I\subset\mR$. Importantly, the bounds solely depend on the length of the interval $I$ and are independent of its location within $\mR$. In addition, we denote the corresponding norms with $X$ replaced by $L^1$ and $H^{\frac32+\kappa}$ for some $\kappa>0$.

	\section{ Stochastic convex integration with pathwise estimates and results}\label{ci2}
	In \cite{CDZ22} and \cite{HZZ22b}, the authors presented a stochastic convex integration method to construct solutions on the entire timeline $\R$. They achieved this by introducing expectations to the iterative estimates in convex integration. In this section, we integrate their methods with pathwise estimates in the context of the stochastic Euler equations. Our objective is to construct H\"{o}lder continuous Euler flows, which are similar to the deterministic case as presented in \cite{BV19}. In that case, the transport equations (see \eqref{eq:te} below) are used to improve regularity. However, this method necessitates pathwise positive lower bounds for the solutions of the transport equations. Moment bounds during the inductive iteration become insufficient. Therefore, we introduce additional pathwise estimates in convex integration schemes. 
	To control the growth of the noise, we employ the cut-off technique to the stochastic convolution. This enables the introduction of pathwise estimates during the inductive iteration. We combine this with moment bounds to derive desired uniform estimates. 
	Compared to only using moment bounds in the inductive iteration in \cite{CDZ22} and \cite{HZZ22b}, the advantage of our approach is that it only requires lower moments of the solutions, and avoids the requirements for higher moments as in \cite{HZZ22b}. This can simplify the computations and enhance the regularity of the solutions.
	
	We intend to develop an iteration procedure leading to the proof of Theorem~\ref{Thm1}. 
	To this end, we decompose stochastic Euler system \eqref{eul1} into two parts, one is linear and involves the noise, whereas the second one is a random PDE. More precisely, we consider the stochastic linear equation
	\begin{equation}\label{li:sto}
		\aligned
		\dif z+z \dif t&=\dif B,
		\\\div z&=0,
		\endaligned
	\end{equation}
	where $B$ is $\mathbb{R}^3$-valued two-sided trace-class $GG^*$-Wiener process with divergence-free and mean zero, and $v$ solves the nonlinear equation
	\begin{equation}\label{nonlinear}
		\aligned
		\partial_t v-z+\div((v+z)\otimes (v+z))+\nabla P&=0,
		\\\div v&=0.
		\endaligned
	\end{equation}
	Here, $z$ is divergence-free by the assumptions on the noise $B$ and we denote the  pressure term associated with $v$ by $P$ .
	
	Furthermore, using the factorization method, it is standard to obtain the regularity of $z$  on a given stochastic basis $(\Omega, \mathcal{F},(\mathcal{F}_{t})_{t\in\mR},\mathbf{P})$ with $(\mathcal{F}_{t})_{t\in\mR}$ being normal filtration given in \cite[Section 2.1]{LR15}. Specifically, the following result follows from \cite[Theorem 5.16]{DPZ92} together with the Kolmogorov continuity criterion.
	\bp\label{eq:Pro1}
	Suppose that $\tr((-\Delta)^{\frac32+\kappa} GG^*)<\infty$ for some $\kappa>0$. Then for any   $\delta \in (0,\frac{1}{2})$, $p\geq 2$
	\begin{equation}\label{zh5/2}
		\sup_{t\in \mR} \mathbf{E}\left[ \|z\|^p_{C_t^{1/2-\delta}{C^\kappa}}\right]\lesssim \sup_{t\in \mR} \mathbf{E}\left[ \|z\|^p_{C_t^{1/2-\delta}H^{\frac32+\kappa}}\right]\leq (p-1)^{p/2}L^p,
	\end{equation}
	where $L\geq 1$ depends on $\tr((-\Delta)^{3/2+\kappa} GG^*)$, $\delta$ and is independent of $p$. 
	\ep
	\begin{proof}
		We recall that the unique stationary solution to \eqref{li:sto} has an explicit form $z(t)=\int_{-\infty}^t e^{s-t}\dif B_s$. According to \cite[Section~2.1]{LR15}, the Wiener process $B$ can be written as $B=\sum_{k\in\mN}\sqrt{c_k } \beta_k e_k$ for an orthonormal basis $\{e_k\}_{k\in\mN}$ of $L^{2}_{\sigma}$ consisting of eigenvectors of $GG^*$ with corresponding eigenvalues $c_k$, and the coefficients satisfy $\sum_{k\in\mN}c_k<\infty$. Here, $\{\beta_{k}\}_{k\in\mN}$ denotes a sequence of mutually independent standard two-sided real-valued Brownian motions. Then it holds for $t\geq s$,
		\begin{align*}
			\\ &\mathbf{E}\|z(t)-z(s)\|^2_{H^{\frac32+\kappa}}
			\\&=\E \left\| (I-\Delta)^{\frac12(\frac32+\kappa)} \int_{s}^{t}e^{r-t} \dif B_r\right\|_{L^2}^2+\E\left\|(I-\Delta)^{\frac12(\frac32+\kappa)} \int_{-\infty}^s (e^{r-t}-e^{r-s})\dif B_r\right\|^2_{L^2}
			\\&=\sum_{k=1}^{\infty}c_k\int_s^t\|(I-\Delta)^{\frac12(\frac32+\kappa)}e_k\|_{L^2}^2 e^{2(r-t)}\dif r 
			\\&\qquad+ \sum_{k=1}^{\infty}c_k\int_{-\infty}^s\|(I-\Delta)^{\frac12(\frac32+\kappa)}e_k\|_{L^2}^2  (e^{r-t}-e^{r-s})^2 \dif r
			\\ &\leq \sum_{k=1}^{\infty}c_k\left(  \|e_k\|_{L^2}^2 + \|(-\Delta)^{\frac12(\frac32+\kappa)}e_k\|_{L^2}^2  \right)  \left( \int_s^t e^{2(r-t)}\dif r +\int_{-\infty}^s (e^{r-t}-e^{r-s})^2 \dif r\right) 
			\\&= (\mathrm{Tr}(GG^*)+\tr((-\Delta)^{3/2+\kappa}GG^*)) \left( \int_s^t e^{2(r-t)}\dif r +\int_{-\infty}^s (e^{r-t}-e^{r-s})^2 \dif r\right) 
			\\& \leq (\mathrm{Tr}(GG^*)+\tr((-\Delta)^{3/2+\kappa}GG^*)) (t-s).
		\end{align*}
		Using Gaussianity we have 
		\begin{align*}
			\mathbf{E}\|z(t)-z(s)\|^p_{H^{\frac32+\kappa}}\leq (p-1)^{\frac{p}{2}}\Big(\E\|z(t)-z(s)\|_{H^{\frac32+\kappa}}^{2}\Big)^{\frac{p}{2}}.
		\end{align*}
		By Kolmogorov’s continuity criterion and fundamental Sobolev embedding $H^{\frac{3}{2}+\kappa} \hookrightarrow C^{\kappa}$, we obtain     for any $\delta \in (0,\frac{1}{2})$ 
		\begin{align*}
			\sup_{t\in \mR} \mathbf{E}\left[ \|z\|^p_{C_t^{1/2-\delta}{C^\kappa}}\right] \lesssim	\sup_{t\in \mR} \mathbf{E}\left[ \|z\|^p_{C_t^{1/2-\delta}H^{\frac32+\kappa}}\right] \leq (p-1)^{p/2}L^p,
		\end{align*}
		where $L\geq 1$ depends on $\tr((-\Delta)^{\frac32+\kappa} GG^*)$, $\delta$ and is independent of $p$.
	\end{proof}
	
	\subsection{Stochastic convex integration set-up and iterative proposition}
	Now, we apply the convex integration method to the nonlinear equation \eqref{nonlinear}. The convex integration iteration is indexed by a parameter $q\in\mathbb{N}_0$. We consider an increasing sequence $\{\lambda_q  \}_{q\in \mathbb{N}_0}$ which diverges to $\infty$, and a bounded sequence $\{\delta_q  \}_{q\in \mathbb{N}}$  which decreases to $0$. We choose $a,b \in \mathbb{N}$ sufficiently large and $\beta \in (0,1)$ sufficiently small and let
	$$\lambda_q=a^{b^{q}}, \quad \delta_1=3rL^2, \quad  \delta_q= \frac12 \lambda_2^{2\beta}\lambda_{q}^{-2\beta}, \ q\geq 2.$$
	 More details on the choice of these parameters will be given in Section~\ref{311'} below. 
	
    At each step $q$, a pair $(v_q ,\mathring{R}_q)$ is constructed to solve the following system
	\begin{equation}\label{euler1}
		\aligned
		\partial_t v_q -z_q+\div((v_q+z_q)\otimes (v_q+z_q))+\nabla p_q&=\div
		\mathring{R}_q,
		\\\div v_q&=0.
		\endaligned
	\end{equation}
	In the above we define 
	\begin{align}\label{def z}
		z_q(t,x)=\chi_{q}\left(\|\tilde{z}_q(t)\|_{C^1_x}\right)\tilde{z}_q(t,x),
	\end{align}
	where $\chi_q$ is a non-increasing smooth function satisfying
	$$\chi_q(x)=
	\begin{cases}
		1, & x\in[0,\lambda_{q+2}^{\gamma}],\\
		0, & x\in(\lambda_{q+3}^{\frac{1}{4}\gamma},\infty).
	\end{cases}
	$$
	Moreover, we suppose $\chi_{q}$ with derivative bounded by $1$. This is achieved by choosing $a$ sufficiently large such that
	\begin{align}\label{zqcut}
		\frac{1}{\lambda_{q+3}^{\gamma/4}-\lambda_{q+2}^\gamma}\ll1.
	\end{align}
	 In addition, the $\tilde{z}_q$ in \eqref{def z} is given by $\tilde{z}_q=\mathbb{P}_{\leq f(q)}z$ with
	$f(q)=\lambda_{q+1}^{\gamma/8}$. Here $\mathbb{P}_{\leq f(q)}$ is the Fourier multiplier operator, which projects a
	function onto its Fourier frequencies $\leq f(q)$ in absolute value, and $\gamma$ is another small parameter which will be fixed in Section~\ref{311'}. 
	$\mathring{R}_q$ on the right hand of \eqref{euler1} is a $3\times 3$ matrix which is trace-free and we put the trace part into the pressure. 
	Keeping \eqref{zh5/2} in mind, by the Sobolev embedding $\|f\|_{L^\infty}\lesssim \|f\|_{H^{\frac32+\kappa}}$ for $\kappa>0$, we observe that for any 
	$p\geq2$ and $\delta \in (0,\frac12)$
	\begin{equation}\label{estimate-zq}
		\aligned
		\$z_q\$_{C^0,p}&\lesssim \$\tilde{z}_q\$_{H^{\frac32+\kappa},p}\lesssim \$z\$_{H^{\frac32+\kappa},p} \lesssim (p-1)^{\frac12}L,
		\\ \$z_q\$_{C^1,p} &\lesssim  \$\tilde{z}_q\$_{C^1,p}\lesssim f(q)\$z\$_{H^{\frac32+\kappa},p} \lesssim (p-1)^{\frac12}L\lambda_{q+1}^{\frac{\gamma}{8}},
		\\ \$z_q\$_{C^{\frac12-\delta}_tC^0,p}&\lesssim \$\tilde{z}_q\$_{C^{\frac12-\delta}_tC^1_x,2p} \$\tilde{z}_q\$_{C^{\frac12-\delta}_tC^0,2p} 
		\lesssim \lambda_{q+1}^{\frac{\gamma}{8}}\$z\$_{C^{\frac12-\delta}_tH^{\frac32+\kappa},2p}^2 \lesssim (2p-1)L^2\lambda_{q+1}^{\frac{\gamma}{8}}.
		\endaligned
	\end{equation}
	The above estimates \eqref{estimate-zq} will be frequently used in the sequel.
	\begin{remark}[On the choice of cut-off]
In order to apply the stationary phase lemma (see Lemma~\ref{lemma1'} below) to derive the desired bound of $\mathring{R}_{q+1}$, it is necessary to have a $pathwise$ positive lower bound for the solution $\Phi_{k,j}$ of equation \eqref{eq:te} below (as shown in \eqref{Phijb}). This requires a pathwise bound of $C^1_x$-norm of $z_q$ part from the noise. Consequently, we truncate the $C^1_x$-norm of $\tilde{z}_q$ by a cut-off function depending on $q$. Although the pathwise $C^1_x$-norm for $z_q$  may blow up, we choose the parameter $\lambda_{q+3}^{\gamma/4}$ in the definition of $\chi_q$ to ensure it does not increase too fast. On the other hand, our technique requires absorbing the difference of $z_{q+1}$ and $z_q$, i.e. $\$z_{q+1}-z_q\$_{C^0,r}$ into the small parameter $\delta_{q+3}$ to derive the inductive estimates for $\mathring{R}_{q+1}$ (see \eqref{vqee} and \eqref{vqf} below). Therefore,  the choice of parameter $\lambda_{q+2}^{\gamma}$ is to bound $\$z_{q+1}-z_q\$_{C^0,r}$ by $\lambda_{q+2}^{-\gamma}$, which can be absorbed into $\delta_{q+3}$ by choosing suitable parameters as stated in Section~\ref{311'} below.
	\end{remark}
	
	Under the above assumptions, the key result is the following iterative proposition, and the proof of this result is presented in Section~\ref{convex} and Section~\ref{s:it1} below.
	\bp \label{p:iteration1} 
	Suppose that $\tr((-\Delta)^{\frac32+\sigma} GG^*)<\infty$ for some $\sigma>0$ and let $r>1$ be fixed. Given a smooth function $e:\R\to(0,\infty)$ such that $\bar e\geq e(t)\geq \underline{e}\geq 6{\cdot48} {\cdot(2\pi)^3}c_R^{-1}L^2$  with $\|e'\|_{C^{0}}\leq \tilde e$ for some constants $\overline{e},\underline{e},\tilde{e}>0$, where $L$ is from Proposition~\ref{eq:Pro1} and $c_R>0$ is a small universal constant (see \eqref{cR} below). There exists a choice of
	parameters $a,b,\beta$ and $\gamma$ such that the following holds true: Let $(v_q ,\mathring{R}_q)$ for some $q\in \mathbb{N}_0$ be an $(\mathcal{F}_t)_{t\geq 0}$-adapted solution to \eqref{euler1} satisfying
	\begin{subequations}\label{it:vq}
		\begin{align}
			\|v_q\|_{C_{t,x}^0} & \leq \lambda_{q+2}^{\gamma}\label{vqaa},
			\\ \$v_q\$_{C^0,2r} & \leq \lambda_q^{\beta}\label{vqbb},
			\\ \|v_q\|_{C_{t,x}^1} & \leq \lambda^{\frac{3}{2}}_q \delta_q^{\frac{1}{2}}, \label{vqcc}
			\\ \|\mathring{R}_q\|_{C^0_{t,x}} & \leq\lambda^{\gamma}_{q+3}\label{vqdd},
			\\ \$\mathring{R}_q\$_{C^0,r} & \leq \delta_{q+1}\label{vqee},
			\\\$\mathring{R}_q\$_{L^1,1}&\leq\frac{1}{48}c_R\delta_{q+2}\underline{e}.\label{vqf}
		\end{align}
	\end{subequations}
	Moreover, for any $t\in \mR$
	\begin{equation}\label{vqg}
		\aligned
		\frac34\delta_{q+1}e(t)\leq e(t)-\E\|(v_q+z_q)(t)\|_{L^2}^2\leq \frac54\delta_{q+1}e(t),\quad & q\geq1,
		\\ 	\E\|v_0(t)+z_0(t)\|^2_{L^2} \leq \frac12e(t), \quad &q=0. 
		\endaligned
	\end{equation}
	Then there exists  an $(\mathcal{F}_t)_{t\geq 0}$-adapted process $(v_{q+1}, \mathring{R}_{q+1})$ which solves \eqref{euler1}, obeys \eqref{vqaa}-\eqref{vqf} and \eqref{vqg} at the level $q+1$ and satisfies
	\begin{equation}\label{vq+1-vq}
		\$v_{q+1}-v_q\$_{C^0,2r}\leq \bar{M}\delta_{q+1}^\frac{1}{2},
	\end{equation}
	where $\bar{M}$ is a universal constant that will be fixed throughout the iteration.
	\ep
	
	We intend to start the iteration from $v_{0}\equiv 0$ on $\mR$. In this case, we have
	$\mathring{R}_0=z_{0}\mathring\otimes z_{0}-\mathcal{R}z_{0}$.
	Recalling the definition of $z_0$ in \eqref{def z}, we obtain
	\begin{align*}
		\|\mathring{R}_0\|_{C^0_{t,x}} \leq \|z_0\|^2_{C^0_{t,x}}+ \|z_0\|_{C^0_{t,x}} \leq \lambda_3^{\frac{1}{2}\gamma}+ \lambda_3^{\frac{1}{4}\gamma} \leq \lambda_3^{\gamma},
	\end{align*}
	where we chose $a$ sufficiently large to absorb the constant in the last inequality. Taking expectation and using \eqref{estimate-zq}, we obtain
	\begin{align*}
		\$\mathring{R}_0\$_{C^0,r}\leq \$z_0\$^2_{C^0,2r}+ \$z_0\$_{C^0,r} \leq \$z\$^2_{H^{\frac32+\sigma},2r}+ \$z\$_{H^{\frac32+\sigma},r} \leq 2rL^2+rL \leq \delta_1,
		\\ \$\mathring{R}_0\$_{L^1,1}\leq (2\pi)^3 \$\mathring{R}_0\$_{C^0,1}\leq 3{\cdot} (2\pi)^3 L^2\leq \frac{1}{96}c_R\underline{e}\leq \frac{1}{48}c_R\delta_2
		e(t),
	\end{align*}
	where we used $6{\cdot}48{ \cdot}(2\pi)^3c_R^{-1}L^2\leq \underline{e}$, $\delta_1=3rL^2$ and $\delta_2=\frac12$.
	Hence, \eqref{vqdd}, \eqref{vqee}  and \eqref{vqf} are satisfied at the level  $q=0$. Similarly, \eqref{vqg} also holds at the level $q=0$, since we have
	\begin{align*}
		\E\|z_0(t)\|^2_{L^2} \leq (2\pi)^3\E\|z_0(t)\|^2_{C^0}\leq (2\pi)^3\E\|z(t)\|^2_{H^{\frac32+\sigma}}\leq(2\pi)^3L^2\leq \frac{1}{2}e(t).
	\end{align*}
	We will see in the following context, as long as the $q=0$ step in \eqref{vqg} holds, $\zeta_0(t)$ in \eqref{zetaq} below is well-defined, thus the iteration can proceed.

	\subsection{Proof of main result}
	Since the first iteration is established, we deduce the following result by using Proposition \ref{p:iteration1}.
	\bt\label{v1}
	Suppose that $\tr((-\Delta)^{\frac32+\sigma} GG^*)<\infty$ for some $\sigma>0$. Let $r>1$ be fixed and given a smooth function $e:\R\to(0,\infty)$ such that $\bar e\geq e(t)\geq \underline{e}\geq 6{\cdot48} {\cdot(2\pi)^3}c_R^{-1}L^2$  with $\|e'\|_{C^{0}}\leq \tilde e$ for some constants $\overline{e},\underline{e},\tilde{e}>0$. There exists an $(\mathcal{F}_t)_{t\in\mR}$-adapted process $u$ which belongs to $C(\mR;C^{\vartheta})$ $\mathbf{P}$-a.s.  for some $\vartheta>0$  and is an analytically weak solution to \eqref{eul1} in the sense of Definition~\ref{d:sol}.  Moreover, the solution satisfies 
	\begin{align}\label{vtheta:11}\$u\$_{C^{\vartheta},2r}<\infty, \end{align} 
	and for all $t\in\mR$
	\begin{align}\label{eq:K2}
		\mathbf{E}\|u(t)\|_{L^2}^2=e(t).
	\end{align}
	There are infinitely many such solutions by choosing different energy functions $e$.
	\et
	
	\begin{proof}
		We start the iteration Proposition \ref{p:iteration1} with the pair $(v_0,\mathring{R}_0)$ and obtain a sequence of solutions $(v_q,\mathring{R}_q)$. By \eqref{vqcc}, \eqref{vq+1-vq} and interpolation we deduce for any $\vartheta\in \left(0,\min{\{\frac{\sigma}{120{\cdot}b^5},\frac{1}{21{\cdot}b^4}\}} \right)  \in (0,\frac{2}{3} \beta)$, the following series is summable
		\begin{align*}
			\sum_{q\geq0}\$v_{q+1}-v_q\$_{C^{\vartheta},2r}
			&\leq \sum_{q\geq0}\$v_{q+1}-v_q\$^{1-\vartheta}_{C^0,2r} \$v_{q+1}-v_q\$^{\vartheta}_{C^1,2r}
			\\ &\lesssim \sum_{q\geq0} \delta_{q+1}^{\frac{1-\vartheta}{2}} \lambda_{q+1}^{\frac{3}{2}\vartheta} \delta_{q+1}^{\frac{1}{2}\vartheta}
			\leq \sqrt{3r}La^{\frac32b\vartheta}+\lambda_2^{\beta}\sum_{q\geq1} \lambda_{q+1}^{\frac{3}{2}\vartheta-\beta} <\infty.
		\end{align*}
		Thus, we may define a limiting function $v=\lim_{q\to \infty}v_q$ which lies
		in $L^{2r}(\Omega;C(\mR,C^{\vartheta}))$. Since $v_q$ is $(\mathcal{F}_t)_{t\in\mR}$-adapted for every $q\in\mathbb{N}_{0}$, the limit
		$v$ is $(\mathcal{F}_t)_{t\in\mR}$-adapted as well. Combining \eqref{estimate-zq} with \eqref{z-p} in the proof of Proposition~\ref{p:iteration1} and using interpolation again, we deduce for the same $\vartheta$ as above and any $p\geq1$
		\begin{align*}
			\sum_{q\geq0}\$z_{q+1}-z_q\$_{C^{\vartheta},p}
			&\leq \sum_{q\geq0}\$z_{q+1}-z_q\$^{1-\vartheta}_{C^0,p} \$z_{q+1}-z_q\$^{\vartheta}_{C^1,p}
			\\ &\lesssim p\sum_{q\geq0} (\lambda_{q+1}^{-\frac{\gamma}{8}\sigma}+\lambda_{q+2}^{-\gamma}\lambda_{q+1}^{\frac{\gamma}{8}})^{1-\vartheta} \lambda_{q+2}^{\frac{\gamma}{8}\vartheta} 
			\\&\lesssim p\sum_{q\geq1} \left( \lambda_{q+1}^{\frac{\gamma}{8}(b\vartheta+\sigma\vartheta-\sigma)}+\lambda_{q+1}^{\gamma(\frac{9}{8}b\vartheta+\frac{1}{8}-b)}\right) <\infty,
		\end{align*}
		where we used $\vartheta<\min{\{\frac{\sigma}{b+\sigma}, \frac{8b-1}{9b}\}}$ in the last inequality. Then, for any $p\geq 1$, we have $\lim_{q\to \infty}z_q=z$ in $L^p(\Omega;C(\mR,C^\vartheta))$.
		Furthermore, it follows from \eqref{vqee} that $\lim_{q\to \infty}\mathring{R}_q=0$ in $L^1(\Omega;C(\mR,C^0))$.
		Thus, $v$ is an analytically  weak solution to \eqref{nonlinear}. Letting $u=v+z$ we obtain an $(\mathcal{F}_{t})_{t\in\mR}$-adapted analytically weak solution to \eqref{eul1} and the estimate \eqref{vtheta:11} for $u$ holds. Finally, \eqref{eq:K2} follows from \eqref{vqg}. By choosing different energy functions, we can obtain infinitely many such solutions. This concludes the proof of Theorem \ref{v1}.
	\end{proof}
	
		\begin{remark}
		It appears challenging to specify the energy pathwise, as done in \cite{HZZ21markov}. Indeed, to obtain the pathwise estimates for energy, our technique requires absorbing the pathwise $C^0_x$-norms of $\mathring{R}_q$ and $z_{q+1}-z_q$ into the small parameter $\delta_{q+2}$, as demonstrated in the proof of Proposition~\ref{proof:en} below. Instead of using stopping time as in \cite{HZZ21markov} to control the noise, we utilize a cut-off function $\chi_q$ here, the support of which also increases with $q$. This makes it not easy to absorb the above two terms into $\delta_{q+2}$.
	\end{remark}

	\section{Convex integration scheme}\label{convex}
	In this section, we aim to construct a pair $(v_{q+1}, \mathring{R}_{q+1})$ to solve the system \eqref{euler1}, given  $(v_{q}, \mathring{R}_{q})$ as in the statement of Proposition~\ref{p:iteration1}.
	First of all, we fix the parameters used during the construction in Section~\ref{311'} and continue with a mollification step in Section~\ref{312'}.  Section~\ref{s:313} is devoted to introducing the new perturbation $w_{q+1}$.  We follow the construction method in  \cite{HZZ22b} and \cite{BV19} to achieve this. More precisely, we construct new  amplitude functions $a_{(\xi)}$ similarly to \cite{HZZ22b}, meanwhile, we incorporate the solutions for transport equations into Beltrami waves to obtain an acceptable transport error as in \cite{BV19}. As pathwise positive lower bounds of the solutions to the transport equations are required, we use the cut-off technique to control the growth of noise and add pathwise estimates of the velocity field $v_q$ in inductive iteration. In Section~\ref{316'}, we show how to construct the new stress $\mathring{R}_{q+1}$. 
		\subsection{Choice of parameters}\label{311'}
	In the sequel, additional parameters will be carefully chosen to verify all the conditions appearing in the estimates below. First, for a fixed integer $b\geq7$, we let $\beta,\gamma \in (0,1)$ be small parameters satisfying 
	\begin{equation*}
		4b^3\gamma+4b^3\beta<1, \qquad 3b\beta <\gamma.
	\end{equation*}
	In the sequel, we also use the following bound
	\begin{equation*}
		2b^2\beta+\beta<\frac{\gamma\sigma}{8},
	\end{equation*}
	which can be obtained by choosing $\gamma=\frac{2}{9b^3}$ and
	$\beta<\min{\{\frac{\sigma}{78{\cdot}b^5},\frac{2}{27{\cdot}b^4}\}}$. 
	The last free parameter $a\in 2^{36b^2\mathbb{N}}$ is chosen sufficiently large such that $a^{\frac{b\gamma}{8}}\in \mathbb{N}$ and  $2<a^{(\frac{b}{4}-1)b^2\gamma}$. This ensures $f(q)\in \mathbb{N}$ and \eqref{zqcut} holds. In the following sections, we increase $a$ to absorb various implicit and universal constants.
	
	\subsection{Mollification}\label{312'} 
	In order to guarantee smoothness throughout the construction, we replace $v_q$ with a mollified velocity field $v_\ell$. For this purpose, we choose a small parameter $\ell= \lambda_q^{-2}$. Let $\{ \phi_\varepsilon \}_{\varepsilon>0}$ be a family of mollifiers on $\mathbb{R}^3$, and $\{ \varphi_\varepsilon \}_{\varepsilon>0}$ be a family of mollifiers with support in $(0,1)$. Here we use the one-sided mollifiers to reserve adaptedness. We define a mollification of $v_q$, $z_q$ and $\mathring{R}_q$ in space and time by
	\begin{align*}
		v_\ell=(v_q*_x\phi_\ell)*_t\varphi_\ell,\qquad
		z_\ell=({z_q}*_x\phi_\ell)*_t\varphi_\ell, \qquad
		\mathring{R}_\ell=(\mathring{R}_q*_x\phi_\ell)*_t\varphi_\ell,
	\end{align*}
	where $\phi_\ell=\frac{1}{\ell^3}\phi(\frac{\cdot}{\ell})$ and $\varphi_\ell=\frac{1}{\ell}\varphi(\frac{\cdot}{\ell})$. By definition, it is easy to see that $z_\ell$, $v_\ell$ and $\mathring{R}_\ell$ are $(\mathcal{F}_t)_{t\in \mR}$ adapted. Then using \eqref{euler1} we obtain that $(v_\ell ,\mathring{R}_\ell)$ obey
	\begin{equation}\label{mollification1}
		\aligned
		\partial_t v_\ell -z_\ell +\div((v_\ell+z_\ell)\otimes (v_\ell+z_\ell))+\nabla p_\ell&=\div (\mathring{R}_\ell+R_\textrm{commutator}),
		\\\div v_\ell &=0,
		\endaligned
	\end{equation}
	where the commutator error $R_{\textrm{commutator}}$ and pressure term $p_\ell$ are given by
	\begin{equation}\label{Com}
		R_{\textrm{commutator}}=(v_\ell+z_\ell)\mathring{\otimes}(v_\ell+z_\ell)-(((v_q+z_q)\mathring{\otimes}(v_q+z_q))*_x\phi_\ell)*_t\varphi_\ell,
	\end{equation}
	\begin{equation*}
		p_\ell=(p_q*_x\phi_\ell)*_t\varphi_\ell-\frac{1}{3}\big(|v_\ell+z_\ell|^2-(|v_q+z_q|^2*_x\phi_\ell)*_t\varphi_\ell).
	\end{equation*}

		\subsection{Perturbation of the velocity}\label{s:313}
	As usual in convex integration schemes, the new velocity field $v_{q+1}$ is a perturbation of $v_\ell$:
	\begin{align*}
		v_{q+1}:=v_\ell+w_{q+1},
	\end{align*}
	where $w_{q+1}:=w_{q+1}^{(p)}+w_{q+1}^{(c)}$ is constructed following \cite{BV19}.
	We proceed with construction by introducing flow maps that satisfy transport equations and time cutoffs.
	\subsubsection{Flow maps and cutoffs}\label{313'}
	In order to have an acceptable transport error in the estimates of $\mathring{R}_{q+1}$, the perturbation $w_{q+1}$ needs to be transported by the flow of the vector field $\partial_t+(v_\ell+z_\ell) \cdot \nabla$. Similar to \cite{BV19}, a natural way to achieve this, is to replace the
	linear phase $\xi \cdot x$ in the definition of the Beltrami wave $W_{\xi, \lambda}$ by the nonlinear phase $\xi \cdot \Phi(x,t)$, where $\Phi$ is transported by the aforementioned vector field. A slight difference from \cite{BV19} is that we extend the definition of $\Phi$ to the whole time line $\mR$.
	
	For any integer $k\in \mathbb{Z}$, we subdivide $[k,k+1]$ into time intervals of size $\ell$ and solve transport equations on these intervals. For $j\in \{0,1,...,\lceil \ell^{-1} \rceil\}$, 
	we define the adapted map $\Phi_{k,j}:\Omega\times{\mR^3}\times[k+(j-1)\ell,k+(j+1)\ell]\to{\mR^3}$ as the $\mT^3$ periodic solution of
	\begin{equation}\label{eq:te}
		\aligned
		(\partial_t+(v_\ell+z_\ell)\cdot\nabla) \Phi_{k,j}&=0,
		\\ \Phi_{k,j}(k+(j-1)\ell,x)&=x.
		\endaligned
	\end{equation}
	We put further details on the $C_x^n$-estimates of $\Phi_{k,j}$ in Appendix~\ref{s:B}. Specifically, the following two estimates which follow from \eqref{eq:a6} and \eqref{eq:Phi0}
	
	\begin{subequations}\label{t1}
		\begin{align}
			&\sup_{t\in[k+(j-1)\ell,k+(j+1)\ell]} \|\nabla \Phi_{k,j}(t)-\Id\|_{C_x^0} \leq 
			\lambda^{-\frac{1}{2}}_q\ll1 \label{Phija},
			\\ \frac12\leq&\sup_{t\in[k+(j-1)\ell,k+(j+1)\ell]} \| \nabla\Phi_{k,j}(t)\|_{C_x^0} \leq 2.\label{Phijb}
		\end{align}
	\end{subequations}
	
	As indicated by the proof in Appendix~\ref{s:B}, the validity of \eqref{Phija} requires that both $v_q$ and $z_q$ have pathwise bounds. Therefore, we cut off $\tilde{z}_q$ and introduce \eqref{vqcc} in the inductive iteration, instead of solely considering the moment bounds in the inductive iteration as in \cite{HZZ22b}. The above estimates play a crucial role in canceling the oscillation error in the new Reynolds stress $\mathring{R}_{q+1}$. In particular, \eqref{Phijb} ensures the validity of stationary phase Lemma~\ref{lemma1'} below in Section~\ref{sss:R}, which requires the positive lower bounds of $\nabla \Phi_{k,j}$. Moreover, we could easily extend $\Phi_{k,j}$ from $\Omega\times \mathbb{R}^3 \times[k+(j-1)\ell,k+(j+1)\ell]$ 
	to $\Omega\times \mathbb{R}^3 \times[k,k+1]$
	in an adapted way, which is still denoted as $\Phi_{k,j}$. The extension is not unique but in the following, we only use the value of $\Phi_{k,j}$ on $[k+(j-1)\ell,k+(j+1)\ell]$.
	
	We also let $\eta$ be a non-negative function supported in $(-1,1)$, which equals to 1 on $(-\frac{1}{4},\frac{1}{4})$ and such that the square of the shifted bump functions
	\begin{equation*}
		\eta_j(t)=\eta(\ell^{-1}t-j)
	\end{equation*}
	form a partition of unity, namely, for all $t\in [0,1]$
	\begin{equation*}
		\sum_j \eta^2_j(t)=1.
	\end{equation*}
	 We then extend the definition of $\eta$ to $\mathbb{R}$. More precisely, let $\eta^{(k)}(t)=\eta(t-k)$, then we know supp$\eta^{(k)} \subset (k-1,k+1)$. Similarly, the shifted bump functions 
	\begin{align*}
		\eta_{k,j}(t)=\eta(\ell^{-1}(t-k)-j)
	\end{align*}
	form a partition of unity
	\begin{equation}\label{chi2}
		\sum_j \eta^2_{k,j}(t)=1,
	\end{equation} 
	for all $t\in [k,k+1]$.
	\subsubsection{Construction of $w_{q+1}$}
	Let us now proceed with the construction of the perturbation $w_{q+1}$.
	To this end, the building blocks of the perturbation are the Beltrami waves presented in \cite[section 5.4]{BV19}, which we recall in Appendix~\ref{Bw2}. Since we use moment bounds of $\mathring{R}_q$ to the iterative estimates in Proposition~\ref{p:iteration1}, we have to adjust the amplitude functions $a_{(\xi)}$ such that we can apply Lemma~\ref{Belw22} in our setting. 
	More precisely, we first define $\rho$ as follows
	\begin{equation}\label{eq:rho3'}
		\rho(t,x):=\sqrt{\ell^2+|\mathring{R}_\ell(t,x)|^2}+c_*\zeta_\ell(t),
	\end{equation}
	\begin{align}\label{zetaq}
		\zeta_q(t):=\frac1{3{\cdot} (2\pi)^3} \Big[ e(t)(1-\delta_{q+2})-\E\|v_q(t)+{z_q}(t)\|_{L^2}^2\Big],
	\end{align}
	and $	\zeta_\ell:=\zeta_q*_t\varphi_\ell$.
	The constant $c_*>0$ is introduced in Lemma~\ref{Belw22} and 
     we take the constant $c_R$ in Proposition~\ref{p:iteration1} satisfying \begin{align}\label{cR}
     		c_R<c_*.
     \end{align}
     
	Now, we define the amplitude functions
	\begin{equation}\label{def axi}
	a_{(\xi)}(t,x)=a_{q+1,j,\xi}(t,x)=c_*^{-\frac{1}{2}} \rho^{\frac12} \eta_{k,j}(t) \gamma_\xi^{(j)} \left(\Id
	- \frac{c_*\mathring{R}_\ell}{\rho}\right), \ t\in [k,k+1],
\end{equation}
	where $\gamma_\xi^{(j)}$ is introduced in Appendix~\ref{Bw2}. By the definition of $\rho$ we know
	\begin{equation*}
		\left\| \Id-\left(\Id- \frac{c_*\mathring{R}_\ell}{\rho}\right)\right\|_{C^0_{t,x}} \leq c_*.
	\end{equation*}
	As a result,
	$\Id-c_*\rho^{-1}\mathring{R}_{\ell}$ lies in the domain of the function $\gamma_{\xi}^{(j)}$ and 
	we deduce from $\eqref{chi2}$ and  Lemma~\ref{Belw22} that
	\begin{equation}\label{-osc}
		c_*^{-1}\rho\Id-\mathring{R}_\ell=\frac{1}{2} \sum_j \sum_{\xi\in\Lambda_j} a^2_{(\xi)}(\Id-\xi\otimes\xi)
	\end{equation}
	holds pointwise, where $\Lambda_j\subset \mathbb{S}^2\cap \mathbb{Q}^3$ is finite subset. By Lemma~\ref{Belw22}, it is sufficient to consider index sets $\Lambda_0$ and $\Lambda_1$ to have 12 elements, and for $j\in \mathbb{Z}$ we denote $\Lambda_j=\Lambda_{j \ \mathrm{mod}\ 2}$.

	With these preparations in hand,  we define the principal part $w_{q+1}^{(p)}$ of the perturbation $w_{q+1}$ as in \cite{BV19}. More precisely, for $t\in [k,k+1]$, we define
	\begin{equation*}
		w_{(\xi)}(t,x):= a_{q+1,j,\xi}(t,x)B_\xi e^{i\lambda_{q+1}\xi\cdot\Phi_{k,j}(t,x)},
	\end{equation*}
	and
	\begin{equation}\label{principal}
		\aligned
		w^{(p)}_{q+1}(t,x):&=\sum_j\sum_{\xi\in\Lambda_j}a_{q+1,j,\xi}(t,x)B_\xi e^{i\lambda_{q+1}\xi\cdot\Phi_{k,j}(t,x)}
		\\ &=\sum_j\sum_{\xi\in\Lambda_j}c_*^{-\frac{1}{2}} \rho^{\frac12} \eta_{k,j}(t) \gamma_\xi^{(j)} \left(\Id- \frac{c_*\mathring{R}_\ell}
		{\rho}\right)B_\xi e^{i\lambda_{q+1}\xi\cdot\Phi_{k,j}(t,x)},
		\endaligned
	\end{equation}
	where $B_{\xi}$ is a suitable vector defined in Appendix~\ref{Bw2}. Since the coefficients $a_{(\xi)}$ and $\Phi_{k,j}$ are $(\mathcal{F}_t)_{t\in \mR}$-adapted, we deduce that
	$w_{q+1}^{(p)}$ is $(\mathcal{F}_t)_{t\in \mR}$-adapted.
	
	Next, we define the incompressibility correction $w_{q+1}^{(c)}$. We aim to add a corrector to $w_{q+1}^{(p)}$ such that the resulting function is perfect $\mathrm{curl}$, making it thus divergence-free. To this end, it is useful to introduce the following scalar $phase$ function
	\begin{equation}
		\phi_{(\xi)}(t,x):=e^{i \lambda_{q+1}\xi\cdot (\Phi_{k,j}(t,x)-x)}, \quad t\in[k,k+1].
	\end{equation}
	Using the same notation as in Appendix~\ref{Bw2}, we define 
	\begin{align*}
		W_{(\xi)}(x)=B_\xi e^{i\lambda_{q+1} \xi\cdot x}.
	\end{align*}
	Since $\phi_{(\xi)}$ and $a_{(\xi)}$ are scalar functions, it follows from Appendix~\ref{Bw2} that $\mathrm{curl}$ $W_{(\xi)}=\lambda_{q+1} W_{(\xi)}$. By a direct computation, we deduce
	\begin{align}\label{curl W}
		a_{(\xi)} \phi_{(\xi)} W_{(\xi)}=\frac{1}{\lambda_{q+1}} \textrm{curl} \left(a_{(\xi)} \phi_{(\xi)} W_{(\xi)}\right)-\frac{1}{\lambda_{q+1}} \nabla \left(a_{(\xi)} \phi_{(\xi)}\right) \times W_{(\xi)}.
	\end{align}
	We therefore define
	\begin{equation}
		\aligned
		w_{(\xi)}^{(c)}(t,x):&=\frac{1}{\lambda_{q+1}} \nabla  (a_{(\xi)}\phi_{(\xi)}) \times B_\xi e^{i\lambda_{q+1}\xi\cdot x}
		\\&=\left( \frac{\nabla a_{(\xi)}}{\lambda_{q+1}}+ i a_{(\xi)} (\nabla \Phi_{k,j}(t,x)-\Id)\xi \right) \times W_{(\xi)}(\Phi_{k,j}(t,x)), \ t\in [k,k+1].
		\endaligned
	\end{equation}
	The incompressibility correction $w^{(c)}_{q+1}$ is defined as
	\begin{equation*}
		w^{(c)}_{q+1}(t,x)=\sum_j\sum_{\xi\in\Lambda_j}w_{(\xi)}^{(c)}(t,x).
	\end{equation*}
	 Since the coefficients $a_{(\xi)}$ and $\Phi_{k,j}$ are $(\mathcal{F}_t)_{t\in \mR}$-adapted, we deduce that
	$w_{q+1}^{(c)}$ is $(\mathcal{F}_t)_{t\in \mR}$-adapted. 
   Moreover, by \eqref{curl W} we deduce that 
		\begin{align*}
			w^{(p)}_{q+1}+w^{(c)}_{q+1}=\frac{1}{\lambda_{q+1}}\sum_j\sum_{\xi\in\Lambda_j}\textrm{curl} \left(a_{(\xi)} \phi_{(\xi)} W_{(\xi)}\right).
	\end{align*}
   
   	Finally, the new perturbation is defined as  
   	\begin{align}\label{wq+1'}
   		w_{q+1}:=w^{(p)}_{q+1}+w^{(c)}_{q+1},
   	\end{align}
   	and so clearly $w_{q+1}$ is  mean zero, divergence-free and $(\mathcal{F}_t)_{t\in \mR}$-adapted.  The new velocity at the level $q+1$ is defined as
   	$$	v_{q+1}:=v_\ell+w_{q+1}.$$
   	Thus, by the previous discussion, it is also divergence-free and $(\mathcal{F}_t)_{t\in \mR }$-adapted.

		\subsection{Definition of the Reynolds stress $\mathring{R}_{q+1}$}\label{316'}
	
Recalling the system \eqref{mollification1}
		\begin{align*}
			\partial_t v_\ell -z_\ell +\div((v_\ell+z_\ell)\otimes (v_\ell+z_\ell))+\nabla p_\ell&=\div (\mathring{R}_\ell+R_\textrm{commutator}),
		\end{align*}
		and substituting $v_{q+1}=v_\ell +w_{q+1}$ into $\eqref{euler1}$ at the level $q+1$, we obtain that
	\begin{equation}\label{Rno1}
		\aligned
		& \div \mathring{R}_{q+1}-\nabla p_{q+1}
		\\& =\underbrace{(\partial_t+(v_\ell+z_\ell)\cdot \nabla)w_{q+1}^{(p)}}
		_{\div(R_{\textrm{transport}})}+ \underbrace{ \div (w_{q+1}^{(p)} \otimes w_{q+1}^{(p)}+\mathring{R}_\ell)}
		_{\div(R_{\textrm{oscillation}})+\nabla p_{\textrm{oscillation}}}
		\\&+\underbrace{w_{q+1}\cdot \nabla(v_\ell +z_\ell)}_{\div(R_{\textrm{Nash}})}+\underbrace{(\partial_t+(v_\ell+z_\ell)\cdot \nabla ) w_{q+1}^{(c)}+\div(w_{q+1}^{(c)}\otimes w_{q+1}+w_{q+1}^{(p)} \otimes w_{q+1}^{(c)})}_{\div(R_{\textrm{corrector}})+\nabla p_{\textrm{corrector}}}
		\\&+\underbrace{\div(v_{q+1} \otimes (z_{q+1}-z_{\ell} )+(z_{q+1} -z_{\ell} )\otimes v_{q+1}+z_{q+1}\otimes z_{q+1}-z_{\ell} \otimes z_{\ell})-z_{q+1}+z_{\ell}}_{\div(R_{\textrm{commutator1}})+\nabla p_{\textrm{commutator1}}}
		\\ &+\div(R_{\textrm{commutator}}) -\nabla p_{\ell}.
		\endaligned
	\end{equation}
	Here  $R_{\textrm{commutator}}$ is defined in $\eqref{Com}$, and by using the inverse divergence operator $\mathcal{R}$ introduced in Section~\ref{s:2.1}, we define
	\begin{equation*}
		\aligned
		R_{\textrm{transport}}&:=\mathcal{R}\left((\partial_t+(v_\ell+z_\ell)\cdot \nabla)w_{q+1}^{(p)}\right),
		\\ R_{\textrm{Nash}}&:=\mathcal{R} ( w_{q+1}\cdot \nabla(v_\ell +z_\ell)),
		\\ R_{\textrm{corrector}}&:=\mathcal{R} \left((\partial_t+(v_\ell+z_\ell)\cdot \nabla ) w_{q+1}^{(c)}\right)+\left(w_{q+1}^{(c)}\mathring \otimes w_{q+1}^{(c)}+w_{q+1}^{(p)}\mathring \otimes w_{q+1}^{(c)}+w_{q+1}^{(c)}\mathring \otimes w_{q+1}^{(p)}\right),
		\\ R_{\textrm{commutator1}}&:=v_{q+1}\mathring \otimes (z_{q+1}- z_{\ell}) +(z_{q+1}-z_{\ell})\mathring \otimes v_{q+1}+z_{q+1}\mathring \otimes z_{q+1}-z_{\ell}\mathring \otimes z_{\ell}-\mathcal{R}(z_{q+1}-z_{\ell}),
		\\ \nabla p_{\textrm{corrector}}&:=\frac{1}{3}(2w_{q+1}^{(c)}\cdot w_{q+1}^{(p)}+|w_{q+1}^{(c)}|^2),
		\\\nabla p_{\textrm{commutator1}}&:=\frac{1}{3}(v_{q+1} \cdot z_{q+1}-v_{q+1} \cdot z_{\ell} +z_{q+1} \cdot v_{q+1}-z_{\ell} \cdot v_{q+1}+z_{q+1} \cdot z_{q+1}-z_{\ell}\cdot z_{\ell}).
		\endaligned
	\end{equation*}
	
	In order to define the remaining oscillation error from the second line in \eqref{Rno1}, we first note that for $j$ and $j'$ such that $|j-j'|\geq2$, we have $\eta_{k,j}(t)\eta_{k,j'}(t)=0$ . Second, we have $\Lambda_j \cap \Lambda_{j'}= \emptyset$ for $|j-j'|=1$, and by Lemma~\ref{Belw11} we have  $\div (W_\xi \otimes W_{\xi'}+W_{\xi'} \otimes W_{\xi})=\nabla(W_{\xi}\cdot W_{\xi'})$. Therefore, we may use Lemma~\ref{Belw22} and \eqref{-osc} to obtain
	
	\begin{equation}\label{Ros}
		\aligned
		\div (w_{q+1}^{(p)} \otimes w_{q+1}^{(p)}+\mathring{R}_\ell)
		&=\div \left(\sum_{j,j^{'} ,\xi,\xi^{'}} w_{(\xi)}\otimes w_{(\xi^{'})}+\mathring{R}_{\ell} \right)
		\\&=\div\left( \sum_{j,\xi}c^{-1}_*\rho \eta_{k,j}^2\gamma_{\xi}^{(j)}\left(\Id
		- c_*\rho^{-1}\mathring{R}_\ell \right)^2 B_{\xi}\otimes B_{-\xi} +\mathring{R}_{\ell}\right) 
		\\&\qquad \qquad \qquad +\div\left( \sum_{j,j',\xi+\xi'\neq0} a_{(\xi)}a_{(\xi')}W_{\xi}\circ\Phi_{k,j}\otimes W_{\xi'}\circ\Phi_{k,j'} \right) 
		\\&=\div\left( c^{-1}_*\rho(\Id
		- c_*\rho^{-1}\mathring{R}_\ell)+ \mathring{R}_\ell \right) 
		\\ &\qquad \qquad \qquad +\div\left( \sum_{j,j',\xi+\xi'\neq0} a_{(\xi)}a_{(\xi')}W_{\xi}\circ\Phi_{k,j}\otimes W_{\xi'}\circ\Phi_{k,j'} \right) 
		\\&=\nabla \left(c_*^{-1}\rho\right)+\sum_{j,j^{'} ,\xi+\xi^{'}\not =0} (W_{(\xi)}\otimes W_{(\xi^{'})}) \nabla \left(a_{(\xi)} a_{(\xi^{'})} \phi_{(\xi)} \phi_{(\xi^{'})}\right)
		\\ &\qquad \qquad \qquad+\frac{1}{2} \sum_{j,j^{'},\xi+\xi^{'}\not =0} a_{(\xi)} a_{(\xi^{'})} \phi_{(\xi)} \phi_{(\xi^{'})} \nabla \left(W_{(\xi)}\cdot W_{(\xi^{'})} \right)
		\\&=\div(R_{\textrm{oscillation}})+\nabla p_{\textrm{oscillation}}.
		\endaligned
	\end{equation}
	Note that compared to \cite{BV19}, a slight difference is the appearance of  $\div(\rho \mathrm{Id})$. This term can be put into the pressure term $p_{\textrm{oscillation}}$. Above, we have denoted 
	\begin{equation*}
		p_{\textrm{oscillation}}:=\frac{1}{2} \sum_{j,j^{'} ,\xi+\xi^{'}\not =0} a_{(\xi)} a_{(\xi^{'})} \phi_{(\xi)} \phi_{(\xi^{'})} \left(W_{(\xi)}\cdot W_{(\xi^{'})} \right)+c_*^{-1}\rho,
	\end{equation*}
	and
	\begin{equation*}
		R_{\textrm{oscillation}}:=\sum_{j,j^{'} ,\xi+\xi^{'}\not =0} \mathcal{R}\left( \left(W_{(\xi)}\otimes W_{(\xi^{'})}-\frac{W_{(\xi)}\cdot W_{(\xi^{'})}}{2}\mathrm{Id}\right) \nabla (a_{(\xi)} a_{(\xi^{'})} \phi_{(\xi)} \phi_{(\xi^{'})})  \right).
	\end{equation*}
	Finally we define the Reynolds stress at the level $q+1$ by
	\begin{equation}\label{Reynold1}
		\mathring{R}_{q+1}=R_{\textrm{transport}}+R_{\textrm{oscillation}}+R_{\textrm{Nash}}+R_{\textrm{corrector}}+R_{\textrm{commutator}}+R_{\textrm{commutator1}},
	\end{equation}
	and 
	\begin{equation*}
		p_{q+1}=p_{\ell}-p_{\textrm{oscillation}}-p_{\textrm{corrector}}-
		p_{\textrm{commutator1}}.
	\end{equation*}

	\section{ Inductive estimates and proof of Proposition~\ref{p:iteration1}}\label{s:it1}
	
	In this section, we collect all the necessary estimates of $v_{q+1}$, $\mathring{R}_{q+1}$ and the energy to complete the proof of the Proposition~\ref{p:iteration1}. As a preliminary step, we need to give estimates on the $C_{t,x}^N$-norm of the amplitude functions $a_{(\xi)}$ defined in Section~\ref{s:313}. 
	 This is done in forthcoming Proposition~\ref{51}, whose proof is given in Appendix~\ref{ap:A'}.
     
     \begin{proposition}\label{51}
     	 Let $a_{(\xi)}$ be defined by \eqref{def axi}. Then we have for any $N\geq2$
     	 \begin{equation}\label{axiN}
     	 	\|a_{(\xi)}\|_{C^N_{t,x}} \lesssim \ell^{-2N-\frac{1}{2}}\lambda_{q+3}^{(N+1)\gamma},
     	 \end{equation}
     	 and	
     	 \begin{equation}\label{axi012}
     	 	\|a_{(\xi)}\|_{C^0_{t,x}} \lesssim \lambda_{q+3}^{\gamma}, \qquad \|a_{(\xi)}\|_{C^1_{t,x}} \lesssim \ell^{-2}\lambda_{q+3}^{2\gamma},
     	 \end{equation}
     	 where the implicit constants are independent of $q$.
     \end{proposition}
	
    This section is organized as follows:
     Section~\ref{sss:v} contains the inductive estimates of $v_{q+1}$, especially the pathwise and moment estimates. We establish the inductive pathwise and moment estimates on $\mathring{R}_{q+1}$ in Section~\ref{sss:R}. Finally, in Section~\ref{s:en'}, we show how the energy is controlled.  Note that we first give pathwise estimates by using \eqref{axiN} and \eqref{axi012}, then taking expectation and using inductive assumptions can derive uniform moment estimates independent  of time. In addition, we also need the estimates of transport equations, as provided in the appendix~\ref{s:B}.  A useful stationary phase lemma on estimating $\mathring{R}_{q+1}$ previously used in the deterministic setting is recalled in Section~\ref{sss:R}.
	
	\subsection{Inductive estimates for $v_{q+1}$}\label{sss:v}
	In this section, we verify the inductive estimates \eqref{vqaa}, \eqref{vqbb}, \eqref{vqcc} and \eqref{vq+1-vq}.
		By the definition of $w^{(p)}_{q+1}$ in \eqref{principal}, it follows from \eqref{vqdd}, \eqref{eq:rho3'} and \eqref{A5} that 
		\begin{equation}\label{wp00}
			\| w^{(p)}_{q+1}\|_{C^0_{t,x}}\lesssim 2|\Lambda_j|c_*^{-\frac{1}{2}}M(\| \mathring{R}_q\|^{\frac{1}{2}}_{C^0_{[t-1,t+1],x}}+\ell^\frac{1}{2}) +2|\Lambda_j| M\bar{e}\delta_{q+1}^\frac12 \leq \lambda_{q+3}^{\frac{3}{4}\gamma},
		\end{equation}
		where $M$ is the universal constant in \eqref{A5}. Taking expectation on both sides of \eqref{wp00}, we use \eqref{vqee} to obtain
		\begin{equation}\label{Ewp0}
			\aligned
			\$ w^{(p)}_{q+1}\$_{C^0,2r} \lesssim 2|\Lambda_j| c_*^{-\frac{1}{2}}M (\$\mathring{R}_q\$_{C^0,r}^{\frac{1}{2}}+\ell^\frac{1}{2}) +2|\Lambda_j| M\bar{e} \delta_{q+1}^\frac12 \leq 
			\frac{1}{4}\bar{M}\delta_{q+1}^{\frac{1}{2}},
			\endaligned
		\end{equation}
		where $\bar{M}$ is a universal constant satisfying $9c_*^{-\frac12}M|\Lambda_j|+8M\bar{e}|\Lambda_j|<\bar{M}$ and we choose $a$ sufficiently large to absorb the constant in the above two inequalities. Using \eqref{Phijb}, \eqref{axi012} and \eqref{E1}, we obtain 
		\begin{equation}\label{wptx1}
			\aligned
			\| w^{(p)}_{q+1}\|_{C^1_{t,x}} 
			& \lesssim \sup_j\sum_{\xi\in\Lambda_j}\|a_{(\xi)}\|_{C^1_{t,x}}+\lambda_{q+1}\|a_{(\xi)}\|_{C^0_{t,x}}\left(\|\partial_t \Phi_{k,j}\|_{C^0}+\|\nabla \Phi_{k,j}\|_{C^0 }\right)
			\\&\lesssim \lambda_q^4\lambda_{q+3}^{2\gamma}+\lambda_{q+3}^{2\gamma}\lambda_{q+1}\leq \frac{1}{4} \lambda_{q+1}^{\frac{3}{2}}\delta_{q+1}^{\frac{1}{2}},
			\endaligned
		\end{equation}
		where we used $b\geq7$ and $2b^3\gamma+b\beta<1$ to have $\lambda_{q}^{2b^3\gamma+4}<\lambda_{q}^{\frac32b-b\beta}$ and $\lambda_{q}^{2b^3\gamma}<\lambda_{q}^{\frac{1}{2}b-b\beta{\tiny }}$ in the last inequality. We also chose $a$ sufficiently large to absorb the constant.

	We next estimate $w^{(c)}_{q+1}$. Using a standard mollification estimate, \eqref{Phija} and \eqref{axi012} we obtain
		\begin{equation}\label{wc0'}
			\aligned
			\|w^{(c)}_{q+1}\|_{C^0_{t,x}} &\lesssim \sup_j\sum_{\xi\in\Lambda_j}\left( \frac{\|\nabla a_{(\xi)}\|_{C^0_{t,x}}}{\lambda_{q+1}} +\|a_{(\xi)}\|_{C^0_{t,x}}\cdot \|\nabla \Phi_{k,j}-\Id\|_{C^0} \right)
			\\&\lesssim \frac{\lambda_q^4\lambda_{q+3}^{2\gamma}}{\lambda_{q+1}}+\lambda_{q+3}^{\gamma}\lambda_q^{-\frac{1}{2}}\leq \frac{1}{4} \delta_{q+1}^{\frac{1}{2}}, 
			\endaligned
		\end{equation}
		where we used $b\geq7$ and $2b^3\gamma+2b\beta<1$ to have $\lambda_{q}^{b^3\gamma}<\lambda_{q}^{\frac12-b\beta}$ and $\lambda_{q}^{4+2b^3\gamma}<\lambda_{q}^{b-b\beta}$ in the last inequality, and $a$ was chosen sufficiently large to absorb the constant. Taking expectation we obtain
		\begin{equation}\label{Ewc0}
			\$w_{q+1}^{(c)}\$_{C^0,2r}\leq \frac{1}{4}\delta_{q+1}^\frac{1}{2}.
		\end{equation}
		Using \eqref{Phija}, \eqref{Phijb}, \eqref{axiN}, \eqref{axi012}, \eqref{wc0'} and estimates \eqref{eq:Phin}--\eqref{E2} for $\Phi_{k,j}$, we obtain 
		\begin{equation} \label{wctx1}
			\aligned
			\|w^{(c)}_{q+1}\|_{C_{t,x}^1}
			\lesssim& \lambda_{q+1}\|w^{(c)}_{q+1}\|_{C^0_{t,x}}\sup_j (\|\partial_t
			\Phi_{k,j}\|_{C^0}+\|\nabla \Phi_{k,j}\|_{C^0})
			\\ &+\sup_j \sum_{\xi \in \Lambda_j} \bigg[
			\frac{\|\nabla a_{(\xi)}\|_{C^1_{t,x}}}{\lambda_{q+1}}
			+\|a_{(\xi)}\|_{C^1_{t,x}}\|\nabla \Phi_{k,j}-\Id\|_{C^0}
			\\ & \qquad \qquad \quad+ \| a_{(\xi)}\|_{C^0_{t,x}}(\|\nabla^2 \Phi_{k,j}\|_{C^0}+\|\partial_t \nabla\Phi_{k,j}\|_{C^0} )\bigg]
			\\ \lesssim& \lambda_{q+1}\lambda_{q+3}^{\gamma}\delta_{q+1}^{\frac{1}{2}} + \frac{\lambda_q^9\lambda_{q+3}^{3\gamma}}{\lambda_{q+1}} + \lambda_q^4\lambda_{q+3}^{2\gamma} \lambda_{q}^{-\frac{1}{2}}+ \lambda_{q+3}^{\gamma}(\ell^{-\frac34}+\lambda_{q}^{\frac32}\lambda_{q+3}^{\gamma})
			\\ \leq& \lambda_{q+1}\lambda_{q+3}^{\gamma}\delta_{q+1}^{\frac{1}{2}} +\lambda_q^4\lambda_{q+3}^{2\gamma}+ 
			\frac{\lambda_q^9\lambda_{q+3}^{3\gamma}}{\lambda_{q+1}} \leq \frac{1}{4}\lambda_{q+1}^{\frac{3}{2}}\delta_{q+1}^{\frac{1}{2}},
			\endaligned
		\end{equation}
		where we used $b\geq7$ and $2b\gamma^3+b\beta<1$ to have $\lambda_{q}^{b^3\gamma}<\lambda_{q}^\frac{b}{2}$, $\lambda_{q}^{4+2b^3\gamma}<\lambda_{q}^{\frac{3}{2}b-b\beta}$ and $\lambda_{q}^{9+3b^3\gamma}<\lambda_{q}^{\frac{5}{2}b-b\beta}$ in the last inequality, and $a$ was chosen sufficiently large to absorb the constant. 
		
		We use a standard mollification estimate and inductive assumption \eqref{vqcc} to obtain for any $t\in \R$
		\begin{equation}\label{vq-vl1}
			\|(v_q-v_\ell)(t)\|_{L^\infty} \lesssim \ell \|v_q\|_{C_{[t-1,t],x}^1} \leq \ell \lambda_q^{\frac{3}{2}} \delta_q^{\frac{1}{2}} \leq \frac14   \delta_{q+1}^{\frac{1}{2}},
		\end{equation}
		where we used $b\beta<1$ and chose $a$ sufficiently large to  absorb the constant. Taking expectation we obtain
		\begin{equation}\label{vl-vq}
			\$v_{\ell}-v_q\$_{C^0,2r}\leq\frac{1}{4}\delta_{q+1}^{\frac{1}{2}}.
		\end{equation}
	
	Combining \eqref{Ewp0}, \eqref{Ewc0} and \eqref{vl-vq} we obtain
	\begin{equation*}
		\aligned
		\$v_{q+1}-v_q\$_{C^0,2r}\leq \$v_\ell-v_q\$_{C^0,2r}+\$w_{q+1}\$_{C^0,2r} \leq \bar{M}\delta_{q+1}^{\frac{1}{2}}.
		\endaligned
	\end{equation*}
	Hence, \eqref{vq+1-vq} holds at the level $q+1$. 
	Moreover, it follows from \eqref{vqaa}, \eqref{wp00} and \eqref{wc0'} that
	\begin{align}\label{vq+10}
		\|v_{q+1}\|_{C^0_{t,x}}\leq \|v_\ell \|_{C^0_{t,x}}+\|w_{q+1}^{(p)}\|_{C^0_{t,x}}+\|w_{q+1}^{(c)}\|_{C^0_{t,x}}\leq \lambda_{q+2}^{\gamma}+\lambda_{q+3}^{\frac{3}{4}\gamma}+\delta_{q+1}^{\frac{1}{2}}\leq \lambda_{q+3}^{\gamma}.
	\end{align}
	By \eqref{vqcc}, \eqref{wptx1} and \eqref{wctx1} we obtain
	\begin{align*}
		\|v_{q+1}\|_{C^1_{t,x}}\leq \|v_\ell \|_{C^1_{t,x}}+\|w_{q+1}^{(p)}\|_{C^1_{t,x}}+\|w_{q+1}^{(c)}\|_{C^1_{t,x}}\leq \lambda^{\frac{3}{2}}_{q}\delta^{\frac{1}{2}}_{q}+\frac12 \lambda^{\frac{3}{2}}_{q+1}\delta^{\frac{1}{2}}_{q+1}\leq \lambda^{\frac{3}{2}}_{q+1}\delta^{\frac{1}{2}}_{q+1}.
	\end{align*}
	By \eqref{vqbb}, \eqref{Ewp0} and \eqref{Ewc0} we obtain 
	\begin{equation}\label{vq+1}
		\$v_{q+1}\$_{C^0,2r}\leq\$v_q\$_{C^0,2r}+\$w^{(p)}_{q+1}\$_{C^0,2r}+\$w^{(c)}_{q+1}\$_{C^0,2r}\leq \lambda_q^{\beta}+\bar{M}\delta_{q+1}^{\frac{1}{2}}\leq\lambda_{q+1}^{\beta}.
	\end{equation}
	Therefore, we have verified that \eqref{vqaa}, \eqref{vqbb} and \eqref{vqcc} hold at the level $q+1$.

	\subsection{Inductive estimates for $\mathring{R}_{q+1}$}\label{sss:R}
	In this section, we shall verify \eqref{vqdd}, \eqref{vqee} and \eqref{vqf} hold at the level $q+1$.  We first proceed with estimates of $R_{\mathrm{commutator1}}$ in subsection~\ref{rcom}. Then we recall a stationary phase lemma in subsection~\ref{spl}. Finally, the remaining errors are controlled in subsection~\ref{462}--\ref{465}.
	\subsubsection{Estimate on $R_{\mathrm{commutator1}}$}\label{rcom}
	First, by the definition of $z_q$ in \eqref{def z}, we have
	\begin{align}\label{zq0}
		\|z_q\|_{C^0_{t,x}}\leq \lambda_{q+3}^{\frac{\gamma}{4}}.
	\end{align}
	Then we use the pathwise estimates \eqref{vq+10} and \eqref{zq0} to obtain the following pathwise bounds
	\begin{equation*}
		\aligned
		\|v_{q+1}\mathring \otimes (z_{q+1}-z_{\ell})\|_{C^0_{t,x}} &\leq \|v_{q+1}\|_{C^0_{t,x}}\| z_{q+1}-z_{\ell}\|_{C^0_{t,x}}\leq \lambda_{q+3}^{\gamma}\lambda_{q+4}^{\frac{\gamma}{4}}+\lambda_{q+3}^{\frac{5}{4}\gamma}\leq \frac{1}{5}\lambda_{q+4}^{\gamma},
		\\ \|z_{q+1}\mathring \otimes z_{q+1}-z_{\ell}\mathring \otimes z_{\ell}\|_{C^0_{t,x}}&\leq\|z_{q+1}\|^2_{C^0_{t,x}}+\|z_{\ell}\|^2_{C^0_{t,x}}\leq\lambda_{q+4}^{\frac{\gamma}{2}}+\lambda_{q+3}^{\frac{\gamma}{2}}\leq \frac{1}{5}\lambda_{q+4}^{\gamma},
		\\\|\mathcal{R}(z_{q+1}-z_{\ell})\|_{C^0_{t,x}}&\leq \|z_{q+1}-z_{\ell}\|_{C^0_{t,x}}\leq  \lambda_{q+4}^{\frac{\gamma}{4}}+\lambda_{q+3}^{\frac{\gamma}{4}}\leq \frac{1}{5}\lambda_{q+4}^{\gamma},
		\endaligned
	\end{equation*}    
	where we chose $a$ sufficiently large to absorb the constant in the above. Summing over the estimates above, the following pathwise bound holds for any $t\in \R$
	\begin{equation}\label{Rcom1tx0}
		\aligned
		\|R_{\textrm{commutator1}}\|_{C^0_{t,x}}\leq 2	\|v_{q+1}\mathring \otimes (z_{q+1}-z_{\ell})\|_{C^0_{t,x}}
		+  \|z_{q+1}&\mathring \otimes z_{q+1}-z_{\ell}\mathring \otimes z_{\ell}\|_{C^0_{t,x}}
		\\&+ \|\mathcal{R}(z_{q+1}-z_{\ell})\|_{C^0_{t,x}}\leq\frac{4}{5}\lambda_{q+4}^\gamma.
		\endaligned
	\end{equation}
	Next we estimate the $\$\cdot\$_{C^0,r}$-norm of $R_{\textrm{commutator1}}$, and we first give an estimate of $z_{q+1}-z_q$.
	\begin{lemma}\label{3.4}
		For any $p\geq1$, we have
		\begin{align}\label{z-p}
		\$z_{q+1}-z_q\$_{C^0,p} \lesssim (\lambda_{q+1}^{-\frac{\gamma\sigma}{8}}+ \lambda_{q+2}^{-\gamma} \lambda_{q+1}^{\frac{\gamma}{8}})p.
		\end{align}
		In particular, for the fixed $r>1$ in Proposition~\ref{p:iteration1}, we have 
		\begin{align}\label{z'}
			\$z_{q+1}-z_q\$_{C^0,2r}\lesssim( \lambda_{q+1}^{-\frac{\gamma\sigma}{8}}+ \lambda_{q+2}^{-\gamma} \lambda_{q+1}^{\frac{\gamma}{8}})r,
		\end{align}
		where the implicit constants are independent of $q$.
	\end{lemma}
	\begin{proof}
		By the definition of $z_q$ in \eqref{def z}, we deduce for any $p\geq1$
		\begin{equation*}
			\aligned
			\$z_{q+1}-z_q\$_{C^0,p} &=\sup_{t\in \mathbb{R}}\mathbf{E}\left(\sup_{s\in[t,t+1]}\|z_{q+1}(s)-z_q(s)\|_{C_x^0}^{p}\right)^{\frac{1}{p}}
			\\ & \leq \sup_{t\in \mathbb{R}}\mathbf{E}\left(\sup_{s\in[t,t+1]}\|\tilde{z}_{q+1}(s)-\tilde{z}_q(s)\|_{C_x^0}^{p} \cdot \chi_{q+1}(\|\tilde{z}_{q+1}(s)\|_{C_x^1})^{p} \right)^{\frac{1}{p}}
			\\& +\sup_{t\in \mathbb{R}} \mathbf{E} \left(\sup_{s\in[t,t+1]} \|\tilde{z}_q(s)\|_{C_x^0}^{p}  \cdot \left|\chi_{q+1}(\|\tilde{z}_{q+1}(s)\|_{C_x^1})- \chi_q(\|\tilde{z}_q(s)\|_{C_x^1})\right|^{p} \right)^\frac{1}{p}
			\\ & =:  \mathrm{\Rmnum{1}+\Rmnum{2}}.
			\endaligned
		\end{equation*}
		We first estimate $\mathrm{\Rmnum{1}}$. Recalling $\tilde{z}_q=\mathbb{P}_{\leq f(q)}z$, we have 
		\begin{align*}
			\tilde{z}_{q+1}(t,x)-\tilde{z}_q(t,x)=\sum_{\lambda_{q+1}^{\gamma/8}< |k|\leq\lambda_{q+2}^{\gamma/8}} e^{ik\cdot x}\hat{z}(t,k),
		\end{align*}
		where $\hat{z}$ is the Fourier transform of $z$, and $k\in \mathbb{Z}^3$. Then using H\"{o}lder's inequality, we deduce for any $t\in \mR$ and previous $\sigma>0$, 
		\begin{align*}
			\|\tilde{z}_{q+1}(t)-\tilde{z}_q(t)\|_{L^\infty} &\leq  \sum_{\lambda_{q+1}^{\gamma/8}< |k|\leq\lambda_{q+2}^{\gamma/8}}  |\hat{z}(t,k)| \\&=\sum_{\lambda_{q+1}^{\gamma/8}< |k|\leq\lambda_{q+2}^{\gamma/8}} (1+|k|^2)^{\frac12(\frac32+\sigma)} |\hat{z}(t,k)| (1+|k|^2)^{-\frac12(\frac32+\sigma)}
			\\ &\leq \left(\sum_{\lambda_{q+1}^{\gamma/8}< |k|\leq\lambda_{q+2}^{\gamma/8}}  (1+|k|^2)^{\frac32+\sigma} |\hat{z}(t,k)|^2 \right) ^\frac12  \left(\sum_{\lambda_{q+1}^{\gamma/8}< |k|\leq\lambda_{q+2}^{\gamma/8}}  (1+|k|^2)^{-(\frac32+\sigma)}  \right) ^\frac12 
			\\&\lesssim\lambda_{q+1}^{-\frac{\gamma\sigma}{8}} \|z(t)\|_{H^{\frac{3}{2}+\sigma}}.
		\end{align*}
		Taking expectation and using \eqref{estimate-zq} we obtain the bound
		\begin{equation}\label{R1}
			\aligned
			\mathrm{\Rmnum{1}}\leq & \sup_{t\in \mathbb{R}}\mathbf{E}\left(\sup_{s\in[t,t+1]}\|\tilde{z}_{q+1}(s)-\tilde{z}_q(s)\|_{C_x^0}^{p}\right)^\frac{1}{p} 
			\\ \lesssim &\lambda_{q+1}^{-\frac{\gamma\sigma}{8}}\sup_{t\in \mathbb{R}}\mathbf{E}\left(\sup_{s\in[t,t+1]}\|z(s)\|_{H^{\frac32+\sigma}}^{p}\right)^\frac{1}{p} \lesssim \lambda_{q+1}^{-\frac{\gamma\sigma}{8}}\$z\$_{H^{\frac32+\sigma},2p}
			\leq \lambda_{q+1}^{-\frac{\gamma\sigma}{8}} \sqrt{2p}L.
			\endaligned
		\end{equation} 
		We next estimate $\mathrm{\Rmnum{2}}$, by the definition of $\chi_{q}$ and $\chi_{q+1}$,  we deduce that 
		\begin{align*}
			| \chi_{q+1}&(\|\tilde{z}_{q+1}(s)\|_{C_x^1})- \chi_q(\|\tilde{z}_q(s)\|_{C_x^1})|
			\\ \leq &
			\mathbf{1}_{\{ \|\tilde{z}_q(s)\|_{C_x^1}\leq \lambda_{q+2}^{\gamma},\ \|\tilde{z}_{q+1}(s)\|_{C_x^1} > \lambda_{q+3}^{\gamma}\} } 
			+ \mathbf{1}_{\{\|\tilde{z}_q(s)\|_{C_x^1}> \lambda_{q+2}^{\gamma} \} } 
			\\ \leq& \mathbf{1}_{\{ \|\tilde{z}_{q+1}(s)\|_{C_x^1} > \lambda_{q+3}^{\gamma} \}} +  \mathbf{1}_{\{ \|\tilde{z}_q(s)\|_{C_x^1} > \lambda_{q+2}^{\gamma} \} }.
		\end{align*}
		Using H\"{o}lder's inequality, we obtain
		\begin{align*}
			\mathrm{\Rmnum{2}} 
			\leq \sup_{t\in \mathbb{R}} \mathbf{E} &\left(\sup_{s\in[t,t+1]} \|\tilde{z}_q(s)\|_{C_x^0}^{p}  \cdot \mathbf{1}_{\{ \|\tilde{z}_{q+1}(s)\|_{C_x^1} > \lambda_{q+3}^{\gamma} \}} \right)^\frac{1}{p}
			\\ &+
			\sup_{t\in \mathbb{R}} \mathbf{E} \left(\sup_{s\in[t,t+1]} \|\tilde{z}_q(s)\|_{C_x^0}^{p}  \cdot  \mathbf{1}_{\{ \|\tilde{z}_q(s)\|_{C_x^1} > \lambda_{q+2}^{\gamma} \} } \right)^\frac{1}{p}
			\\ \leq \sup_{t\in \mathbb{R}} \mathbf{E} &\left( \sup_{s\in[t,t+1]} \|\tilde{z}_q(s)\|_{C_x^0}^{2p}  \right)^{\frac{1}{2p}}  \mathbf{P}\left( \|\tilde{z}_{q+1}\|_{C^0_{[t,t+1]}C^1_x}> \lambda_{q+3                                                                                                                                                                                                                                                                                                                                                                                                                                                                                                                                                                                                                                                                                                                                                                                                                                                                                                                                                                                                                                                                                                                                                                                                                                                                                                                                                                                                                                                                                                                                                                                                                                                                                                                                                                                                                                                                                                                      }^{\gamma} \right)^{\frac{1}{2p}} 
			\\ &+\sup_{t\in \mathbb{R}} \mathbf{E}\left( \sup_{s\in[t,t+1]} \|\tilde{z}_q(s)\|_{C_x^0}^{2p}  \right)^{\frac{1}{2p}}  \mathbf{P}\left( \|\tilde{z}_q\|_{C^0_{[t,t+1]}C^1_x} > \lambda_{q+2}^{\gamma} \right)^{\frac{1}{2p}}. 
		\end{align*}
		By Chebyshev's inequality and \eqref{estimate-zq}, we have the following bound
		\begin{align*}
			\sup_{t\in \mathbb{R}} \mathbf{P}\left(\|\tilde{z}_q\|_{C^0_{[t,t+1]}C^1_x}>\lambda_{q+2}^{\gamma}\right)^{\frac{1}{2p}}
			&\leq \lambda_{q+2}^{-\gamma} \sup_{t\in \mathbb{R}}\left( \mathbf{E}\|\tilde{z}_q\|^{2p}_{C_{[t,t+1]}^0C_x^1}\right)^\frac{1}{2p}
			\\ & \lesssim
			\lambda_{q+2}^{-\gamma} \lambda_{q+1}^{\frac{\gamma}{8}} \$z\$_{H^{\frac32+\sigma},2p}\leq \sqrt{2p}\lambda_{q+2}^{-\gamma} \lambda_{q+1}^{\frac{\gamma}{8}}L.
		\end{align*}
		Similarly as above, we have 
		\begin{equation*} 
			\aligned
			\sup_{t\in \mathbb{R}}   \mathbf{P}\left(\|\tilde{z}_{q+1}\|_{C^0_{[t,t+1]}C^1_x}> \lambda_{q+3}^{\gamma}\right)^{\frac{1}{2p}}
			& \leq \lambda_{q+3}^{-\gamma} \sup_{t\in \mathbb{R}}\left( \mathbf{E}\|\tilde{z}_{q+1}\|^{2p}_{C_{[t,t+1]}^0C_x^1}\right)^\frac{1}{2p} 
			\\ & \lesssim
			\lambda_{q+3}^{-\gamma} \lambda_{q+2}^{\frac{\gamma}{8}} \$z\$_{H^{\frac32+\sigma},2p}\leq \sqrt{2p} \lambda_{q+3}^{-\gamma} \lambda_{q+2}^{\frac{\gamma}{8}}L.
			\endaligned
		\end{equation*}
		With the above estimates in hand, using \eqref{estimate-zq} again, we obtain 
		\begin{equation}\label{R2}
			\aligned
			\mathrm{\Rmnum{2}}& \leq  \$\tilde{z}_q\$_{C^0,2p}( \sqrt{2p}\lambda_{q+3}^{-\gamma} \lambda_{q+2}^{\frac{\gamma}{8}}L+ \sqrt{2p}\lambda_{q+2}^{-\gamma} \lambda_{q+1}^{\frac{\gamma}{8}}L)
			\lesssim pL^2\lambda_{q+2}^{-\gamma} \lambda_{q+1}^{\frac{\gamma}{8}},
			\endaligned
		\end{equation}
		where we used $9b\gamma \leq 8b^2\gamma+\gamma$ to have $\lambda_{q+3}^{-\gamma} \lambda_{q+2}^{\gamma/8}\leq \lambda_{q+2}^{-\gamma} \lambda_{q+1}^{\gamma/8}$.
		Finally, combining \eqref{R1} with \eqref{R2} implies 
		\begin{align*}
			\$z_{q+1}-z_q\$_{C^0,p}=\mathrm{\Rmnum{1}+\Rmnum{2}}\lesssim( \lambda_{q+1}^{-\frac{\gamma\sigma}{8}}+ \lambda_{q+2}^{-\gamma} \lambda_{q+1}^{\frac{\gamma}{8}})pL^2.
		\end{align*}
		In particular, taking $p=2r$ impies \eqref{z-p}. Hence, the proof of Lemma~\ref{3.4} is complete.
	\end{proof}
	
	Next, we estimate $\$R_{\textrm{commutator1}}\$_{C^0,r}$. Using \eqref{vq+1}, \eqref{z'} and H\"{o}lder's inequality, we have 
	\begin{equation*}\label{C11}
		\$v_{q+1} \mathring{\otimes} (z_{q+1}-z_q)\$_{C^0,r} \leq \$v_{q+1}\$_{C^0,p}\$z_{q+1}-z_q\$_{C^0,2r} \lesssim \lambda_{q+1}^{\beta}(\lambda_{q+1}^{-\frac{\gamma\sigma}{8}}+ \lambda_{q+2}^{-\gamma} \lambda_{q+1}^{\frac{\gamma}{8}}).
	\end{equation*}
	Combining \eqref{zh5/2}, \eqref{estimate-zq} with \eqref{z'} and using H\"{o}lder's inequality we obtain
	\begin{equation*}\label{C12}
		\aligned
		\$z_{q+1}\mathring \otimes (z_{q+1}-z_q)\$_{C^0,r} &\leq
		\$z_{q+1}\$_{C^0,2r}\$z_{q+1}-z_q\$_{C^0,2r} 
		\\ &\lesssim \$z\$_{H^{\frac32+\sigma},2r} \$z_{q+1}-z_q\$_{C^0,2r} 
		\lesssim\lambda_{q+1}^{-\frac{\gamma\sigma}{8}}+ \lambda_{q+2}^{-\gamma} \lambda_{q+1}^{\frac{\gamma}{8}}.
		\endaligned
	\end{equation*}
	We use a standard mollification estimate to have for any $\delta \in (0,\frac{1}{12})$
	\begin{align*}
		\|z_\ell(t)-z_q(t)\|_{L^\infty}\lesssim \ell\|z_q\|_{C^0_{[t-1,t]}C^1_x}+\ell^{\frac12-\delta}\|z_q\|_{C^{1/2-\delta}_{[t-1,t]}C^0_x}.
	\end{align*}
	Taking expectation on both sides and using \eqref{estimate-zq}, we deduce
	\begin{equation}\label{zl-zq}
		\aligned
		\$z_\ell-z_q\$_{C^0,2r} &\lesssim \ell \sup_{t\in \mathbb{R}}\left( \mathbf{E}\|z_q\|^{2r}_{C_{[t-1,t+1]}^0C_x^1}\right)^\frac{1}{2r}+ \ell^{\frac{1}{2}-\delta} \sup_{t\in \mathbb{R}}\left( \mathbf{E}\|z_q\|^{2r}_{C_{[t-1,t+1]}^{1/2-\delta}C_x^0}\right)^\frac{1}{2r}
		\\&\lesssim \ell^{\frac{5}{12}}\lambda_{q+1}^{\frac{\gamma}{8}} \sup_{t\in \mathbb{R}}\left( \mathbf{E}\|z\|^{2r}_{C_{[t-1,t+1]}^{1/2-\delta}H^{\frac32+\sigma}}\right)^\frac{1}{2r}\lesssim \lambda_{q+1}^{\gamma}\lambda_{q}^{-\frac{3}{4}}\sqrt{2r}L.
		\endaligned
	\end{equation} 
	Hence, using \eqref{vq+1}, \eqref{zl-zq} and H\"{o}lder's inequality, we have
	\begin{equation*}
		\$v_{q+1}\mathring \otimes (z_\ell-z_q)\$_{C^0,r} \leq \$v_{q+1}\$_{C^0,2r} \$z_\ell-z_q\$_{C^0,2r}
		\lesssim \lambda_{q+1}^{\beta+\gamma}\lambda_{q}^{-\frac{3}{4}}.
	\end{equation*}
	Similarly, we use \eqref{estimate-zq} and \eqref{zl-zq} again to deduce
	\begin{align*}\label{vz11}
		\$z_{q+1}\mathring \otimes (z_\ell-z_q)\$_{C^0,r} &\leq \$z_{q+1}\$_{C^0,2r} \$z_\ell-z_q\$_{C^0,2r} 
		\\ &\lesssim    \$z\$_{H^{\frac32+\sigma},2r}\$z_\ell-z_q\$_{C^0,2r} 
		\lesssim \lambda_{q+1}^{\gamma}\lambda_{q}^{-\frac{3}{4}}.
	\end{align*}
	Summarizing the estimates above, we obtain
	\begin{equation}\label{esti:Rcom1}
		\aligned
		\$R_{\textrm{commutator1}}\$_{C^0,r}&\leq  \$v_{q+1}\mathring \otimes (z_\ell-z_q)\$_{C^0,r} +  \$v_{q+1}\mathring \otimes (z_{q+1}-z_q)\$_{C^0,r}+\$z_{\ell}-z_q\$_{C^0,r}
		\\&+\$z_{q+1}\mathring \otimes (z_\ell-z_q)\$_{C^0,r}+ \$z_{q+1}\mathring \otimes (z_{q+1}-z_q)\$_{C^0,r}+ \$ z_{q+1}-z_q\$_{C^0,r}
		\\&\lesssim \lambda_{q+1}^{\beta+\gamma}\lambda_{q}^{-\frac{3}{4}}+\lambda_{q+1}^{\gamma}\lambda_{q}^{-\frac{3}{4}}+\lambda_{q+1}^{\beta}(\lambda_{q+1}^{-\frac{\gamma\sigma}{8}}+ \lambda_{q+2}^{-\gamma} \lambda_{q+1}^{\frac{\gamma}{8}}) 
		\\&\leq \frac{1}{48}c_R\delta_{q+3}\leq \frac{1}{6}\delta_{q+2},
		\endaligned
	\end{equation}
	and \begin{align*}
		\$ R_{\textrm{commutator1}}\$_{L^1,1}\leq \frac{1}{48}c_R\delta_{q+3}\underline{e},
	\end{align*}
	where we used $2b^2\beta+\beta<\frac{\gamma\sigma}{8}$ to have $\lambda_{q}^{b\beta-\frac{b}{8}\gamma\sigma}<\lambda_{q}^{-2b^3\beta}$. We also used $3b\beta<\gamma$, $4b^3\beta+2b\gamma<1$ to have  $\lambda_{q}^{b\beta+b\gamma-b^2\gamma}<\lambda_{q}^{-2b^3\beta}$ and  $\lambda_{q}^{b\gamma+b\beta-\frac34}<\lambda_{q}^{-2b^3\beta}$ in the above two inequalities and chose $a$ sufficiently large to absorb the constant.
	
	\subsubsection{ Stationary phase Lemma}\label{spl}
	Before we estimate the remaining terms in \eqref{Reynold1}, we first introduce the following lemma which makes rigorous the fact that $\mathcal{R}$ obeys the same elliptic regularity estimates as $|\nabla|^{-1}$. We recall the following stationary phase lemma (see for example \cite[Lemma 5.7]{BV19} and \cite [Lemma 2.2]{DS17}), adapted to our setting. 
	\begin{lemma}\label{lemma1'}
		Given $\xi\in \mathbb{S}^2\cap \mathbb{Q}^3$, let $\lambda \xi \in \mathbb{Z}^3$ and $\alpha \in (0,1)$. Assume that $a\in C^{m,\alpha}(\mathbb{T}^3)$ and $\Phi\in C^{m,\alpha}(\mathbb{T}^3;\mathbb{R}^3)$ are
		smooth functions such that the phase function $\Phi$ obeys 
	   \begin{align*}
	   	C^{-1}\leq |\nabla \Phi |\leq C
	   \end{align*}
		on $\mathbb{T}^3$, for some constant $C\geq 1$. Then, with the inverse divergence operator $\mathcal{R}$ defined in Section~\ref{s:2.1} we have
		for any $m\in \mathbb{N}$
		\begin{align*}
			\left| \int_{\mathbb{T}^3} a(x)e^{i\lambda \xi \cdot \Phi(x)} \dif x\right| \lesssim \frac{\|a\|_{C^m}+\|a\|_{C^0} \|\nabla \Phi\|_{C^m}}{\lambda^m},
		\end{align*}
		$$\left\|\mathcal{R}\left(a(x)e^{i\lambda \xi\cdot \Phi(x)}\right)\right\|_{C^\alpha}\lesssim \frac{\|a\|_{C^0}}{\lambda^{1-\alpha}}+\frac{\|a\|_{C^{m,\alpha}}+\|a\|_{C^0}\|\nabla \Phi\|_{C^{m,\alpha}}}{{\lambda^{m-\alpha}}},$$
		where the implicit constant depends on $C,\alpha$ and m (in particular, not on the frequency $\lambda$).
	\end{lemma}
	
	The above lemma is used to estimate the $\| \cdot \|_{C^0}$ and $\$ \cdot \$_{C^0,r}$-norm of the remaining terms in \eqref{Reynold1}. Indeed, for a fixed $t\in [k+(j-1)\ell,k+(j+1)\ell]$, in view of the bound \eqref{Phijb}, we have $\frac{1}{2}\leq \|\Phi_{k,j}(\cdot,t)\|_{C^0_x} \leq 2$ on $\mathbb{T}^3$. Thus, it follows from Lemma~\ref{lemma1'} that if $a$ is a smooth periodic function, then we have for any $n\in \mathbb{N}$
	\begin{equation}\label{InR}
		\left\|\mathcal{R}(aW_{(\xi)}\circ \Phi_{k,j} )\right\|_{C^0_tC^\alpha_x}  \lesssim  \frac{\|a\|_{C^0_{t,x}}}{\lambda_{q+1}^{1-\alpha}}+\frac{\|a\|_{C^0_tC^{n,\alpha}_x}+\|a\|_{C^0_{t,x}}\|\nabla \Phi_{k,j}\|_{C^0_tC^{n,\alpha}_x}}{{\lambda_{q+1}^{n-\alpha}}}.
	\end{equation} 
	
	In addition, the same proof that was used in \cite{DS17} to prove Lemma~\ref{lemma1'} gives other useful estimates. Let $\xi \in \Lambda_j$ and $\xi' \in \Lambda_{j'}$ for $|j-j'|\leq1$ be such that $\xi+\xi'\neq0$, we have that $|\xi+\xi'|\geq \bar{c}>0$, for some universal constant $\bar{c}\in (0,1)$. Then for a smooth periodic function $a$, we have the following estimates (see \cite[(5.37)]{BV19}), for $n\in \mathbb{N}$
	
	\begin{equation}\label{R1a}
		\left| \int_{\mathbb{T}^3} a(x) e^{i\lambda (\xi \cdot \Phi_{k,j}+ \xi' \cdot \Phi_{k,j'})} \dif x\right| \lesssim \frac{\|a\|_{C^0_tC^n_x}+\|a\|_{C^0_{t,x}} (\|\nabla \Phi_{k,j}\|_{C^0_tC_x^n}+  \|\nabla\Phi_{k,j'}\|_{C^0_tC_x^n})}{\lambda^n},
	\end{equation}
	and
	\begin{equation}\label{R2a}
		\aligned
		\|\mathcal{R}(a(W_{(\xi)}\circ\Phi_{k,j}\otimes &W_{(\xi')}\circ\Phi_{k,j'}))\|_{C^0_tC_x^\alpha} 
		\\ &\lesssim\frac{\|a\|_{C^0_{t,x}}}{\lambda^{1-\alpha}}+\frac{\|a\|_{C^0_tC_x^{n,\alpha}}+\|a\|_{C^0_{t,x}}(\|\nabla \Phi_{k,j}\|_{C^0_tC_x^{n,\alpha}}+  \|\nabla\Phi_{k,j'}\|_{C^0_tC_x^{n,\alpha}})}{{\lambda^{n-\alpha}}}.
		\endaligned
	\end{equation}
	The implicit constant depends only on $\alpha$ and $n$.
	
	In the following subsections, we aim to use the above discussion to control the $\| \cdot \|_{C^0}$ and $\$ \cdot \$_{C^0,r}$ norms, respectively, of the transport error, oscillation error, Nash error and corrector error. We first fix an integer $b\geq7$, then choose $\beta$, $\alpha$ sufficiently small such that $4b^3\gamma+4b^3\beta +2b\alpha<1$  and choose $m\in \mathbb{N}$ such that $m>\frac{2b+5}{b-5}$, at last we choose $a$ sufficiently large to absorb the constant.
	
	\subsubsection{Estimate on the transport  error.}\label{462}Inspecting the definition of $w_{q+1}^{(p)}$ as  \eqref{principal}, we note that the material derivative $\partial_t+(v_\ell+z_\ell)\cdot \nabla$ cannot lead to higher frequency term, since the term $W_{(\xi)} \circ\Phi_{k,j}$ is perfectly transported by $v_\ell+z_\ell$. Therefore, for $t\in [k,k+1]$, $k\in\mathbb{Z}$ we have
	\begin{equation*}
		(\partial_t+(v_\ell+z_\ell)\cdot \nabla)w_{q+1}^{(p)}=\sum_{j,\xi}(\partial_t+(v_\ell+z_\ell)\cdot \nabla)a_{(\xi)}W_{(\xi)} \circ \Phi_{k,j}.
	\end{equation*}
	Recalling the $C^N_{t,x}$-norm estimate of $a_{(\xi)}$ in \eqref{axiN} and \eqref{axi012} and the definition of $z_q$ in \eqref{def z}, we use chain rule and \eqref{vqaa} to obtain that for any $n\in  \mathbb{N}_0$
	\begin{equation}\label{tranerror}
		\aligned
		\|(\partial_t+(v_{\ell}+z_{\ell})\cdot \nabla)a_{(\xi)}\|_{C^n_{t,x}} 
		&\lesssim \|a_{(\xi)}\|_{C^{n+1}_{t,x}} +\sum_{k=0}^{n} \|v_\ell +z_\ell\|_{C^k_{t,x}} \cdot\| \nabla a_{(\xi)}\|_{C^{n-k}_{t,x}}
		\\ &\lesssim \ell^{-2n-3}\lambda_{q+3}^{(n+2)\gamma} + \sum_{k=0}^{n} \ell^{-k}\lambda_{q+3}^{\gamma} \ell^{-2(n-k+1)-1}\lambda_{q+3}^{(n-k+2)\gamma}
		\\ &\lesssim \ell^{-2n-3}\lambda_{q+3}^{(n+3)\gamma}.
		\endaligned
	\end{equation}
	By the $C^n$-norm estimate for $\Phi_{k,j}$ in \eqref{eq:Phin}, namely $\|\nabla \Phi_{k,j}(t)\|_{C^n_x}\leq \ell^{-n}$ for $n\geq1$, applying $\eqref{InR}$ to $R_{\textrm{transport}}$ with  $n=m$ and $a$ replaced by $(\partial_t+(v_{\ell}+z_{\ell})\cdot \nabla)a_{(\xi)}$, we have
	\begin{equation*}\label{Rtran}
		\aligned
		\|R_{\textrm{transport}}\|_{C^0_{t,x}}&\lesssim
		\frac{\ell^{-3}\lambda_{q+3}^{3\gamma}}{\lambda_{q+1}^{1-\alpha}} +\frac{\ell^{-2m-5}\lambda_{q+3}^{(m+4)\gamma}+\ell^{-4-m}\lambda_{q+3}^{3\gamma}}{\lambda_{q+1}^{m-\alpha}}
		\\&\lesssim \frac{\ell^{-3}\lambda_{q+3}^{3\gamma}}{\lambda_{q+1}^{1-\alpha}}\left( 1+\frac{\ell^{-2m-2}\lambda_{q+3}^{(m+1)\gamma}}{\lambda_{q+1}^{m-1}} \right)
		\\& \leq \frac{1}{6\cdot48}c_R\delta_{q+3} ,
		\endaligned
	\end{equation*}
	where we used $b\geq7$ and $3b^3\gamma+2b^3\beta+b\alpha<1
	$ to have $\lambda_{q}^{6+3b^3\gamma}<\lambda_{q}^{b-b\alpha-2b^3\beta}$. We also used $5m+b+5<bm$ to have $\lambda_{q}^{4m+4+(m+1)b^3\gamma}<\lambda_{q}^{(m-1)b}$, and $a$ was chosen sufficiently large to absorb the constant.
	Then taking expectation we obtain
	\begin{align*}
		\$R_{\textrm{transport}}\$_{C^0,r}\leq  \frac16  \delta_{q+2}, \qquad
		\$R_{\textrm{transport}}\$_{L^1,1}\leq  \frac{1}{6\cdot48}  c_R\delta_{q+3}\underline{e}.
	\end{align*}
	\subsubsection{Estimate on the oscillation error.}\label{463}
	Based on the previous computations in \eqref{Ros}, we can rewrite oscillation error as follows, for $t\in [k,k+1]$, $k\in\mathbb{Z}$
	\begin{equation*}
		\aligned
		R_{\textrm{oscillation}}=\sum_{j,j^{'} ,\xi+\xi^{'}\not =0} \mathcal{R} &\bigg[\left(W_{(\xi)}\circ\Phi_{k,j}\otimes W_{(\xi^{'})}\circ\Phi_{k,j^{'}}-\frac{1}{2}(W_{(\xi)}\circ\Phi_{k,j}\cdot W_{(\xi^{'})}\circ\Phi_{k,j^{'}})\Id\right)
		\\& \times \left(\nabla(a_{(\xi)}a_{(\xi^{'})})+i\lambda_{q+1}a_{(\xi)}a_{(\xi^{'})}((\nabla \Phi_{k,j}-\Id)\cdot \xi +(\nabla \Phi_{k,j^{'}}-\Id)\cdot \xi^{'})\right)\bigg].
		\endaligned
	\end{equation*}
	Then using \eqref{axiN}, \eqref{axi012} and chain rule, we have for any $n\in \mathbb{N}_0$
	\begin{equation}\label{aan}
		\aligned
		\|\nabla(a_{(\xi)}a_{(\xi^{'})})\|_{C^n_{t,x}}& \lesssim \sum_{k=0}^{n+1} \|a_{(\xi)}\|_{C^k_{t,x}} \|a_{(\xi)}\|_{C^{n-k+1}_{t,x}} 
		\\ &\lesssim \sum_{k=0}^{n+1} \ell^{-2k-\frac{1}{2}}\lambda_{q+3}^{(k+1)\gamma}\ell^{-2(n-k+1)-\frac{1}{2}}\lambda_{q+3}^{(n-k+2)\gamma}
		\\ &\lesssim \ell^{-2n-3}\lambda_{q+3}^{(n+3)\gamma}.
		\endaligned
	\end{equation}
	Using \eqref{Phija} and \eqref{axi012}, we find
	\begin{equation}\label{Ros1}
		\aligned
		\lambda_{q+1}\|a_{(\xi)}a_{(\xi^{'})}((\nabla \Phi_{k,j}-\Id)\cdot \xi +(\nabla \Phi_{k,j^{'}}-\Id)\cdot \xi^{'})\|_{C^0_{t,x}}\lesssim \lambda_{q+1}\lambda_q^{-\frac{1}{2}}\lambda_{q+3}^{2\gamma}.
		\endaligned
	\end{equation}
	Using \eqref{Phija}, \eqref{axiN}, \eqref{axi012}, \eqref{aan}, \eqref{eq:Phin} and chain rule, we obtain for $n\geq1$
	\begin{equation}\label{Ros2}
		\aligned
		&\lambda_{q+1}\|a_{(\xi)}a_{(\xi^{'})}((\nabla \Phi_{k,j}-\Id)\cdot \xi +(\nabla \Phi_{k,j^{'}}-\Id)\cdot \xi^{'})\|_{C^0_tC^n_x}
		\\ \lesssim& \lambda_{q+1} \sum_{l=0}^{n} \|a_{(\xi)}a_{(\xi')}\|_{C^0_tC^l_x}\|(\nabla \Phi_{k,j}-\Id)\cdot \xi +(\nabla \Phi_{k,j^{'}}-\Id)\cdot \xi^{'}\|_{C^0_tC^{n-l}_x}
		\\ \lesssim &\lambda_{q+1} \sum_{l=0}^{n} \ell^{-2l-1}\lambda_{q+3}^{(l+2)\gamma}\ell^{-(n-l)}
		\lesssim \lambda_{q+1} \ell^{-2n-1}\lambda_{q+3}^{(n+2)\gamma},
		\endaligned
	\end{equation}
	where we used $C^n$-norm estimate of $\nabla \Phi_{k,j}$ \eqref{eq:Phin} in the second inequality, namely, $\|\nabla \Phi_{k,j}(t)\|_{C^n_x}\leq \ell^{-n}$ for $n\geq1$.
	Combining \eqref{aan}, \eqref{Ros1}, \eqref{Ros2} with \eqref{eq:Phin} again, applying \eqref{R2a} with $n=m$ and $a$ replaced by $\nabla(a_{(\xi)}a_{(\xi^{'})})+i\lambda_{q+1}a_{(\xi)}a_{(\xi^{'})}((\nabla \Phi_{k,j}-\Id)\cdot \xi +(\nabla \Phi_{k,j^{'}}-\Id)\cdot \xi^{'})$, we obtain
	\begin{equation*}\label{Rosc}
		\aligned
		\|R_{\textrm{oscillation}}\|_{C^0_{t,x}} \lesssim & \frac{\ell^{-3}\lambda_{q+3}^{3\gamma}}{\lambda_{q+1}^{1-\alpha}}
		+\frac{\ell^{-2m-5}\lambda_{q+3}^{(m+4)\gamma}+\ell^{-m-4}\lambda_{q+3}^{3\gamma}}{\lambda_{q+1}^{m-\alpha}}
		\\&\qquad+\frac{\lambda_{q+1}\lambda_q^{-\frac{1}{2}}\lambda_{q+3}^{2\gamma}}{\lambda_{q+1}^{1-\alpha}} +\frac{ \ell^{-2m-3}\lambda_{q+3}^{(m+3)\gamma}+\ell^{-m-2}\lambda_{q+3}^{2\gamma} }{\lambda_{q+1}^{m-1-\alpha}}
		\\ \lesssim& \frac{\ell^{-3}\lambda_{q+3}^{3\gamma}}{\lambda_{q+1}^{1-\alpha}}\left( 1+ \frac{\ell^{-2m-2}\lambda_{q+3}^{(m+1)\gamma}}{\lambda_{q+1}^{m-2}}\right) +\frac{\lambda_{q+3}^{2\gamma}}{\lambda_q^{\frac{1}{2}}\lambda_{q+1}^{-\alpha}}
		\\\leq & \frac{1}{6\cdot48}  c_R\delta_{q+3},
		\endaligned
	\end{equation*}
	where we used $b\geq7$ and $4b^3\gamma+4b^3\beta+2b\alpha<1$ to have $\lambda_{q}^{6+3b^3\gamma}<\lambda_{q}^{b-b\alpha-2b^3\beta}$ and $\lambda_{q}^{2b^3\gamma}<\lambda_{q}^{\frac12-b\alpha-2b^3\beta}$. We also used $5m+2b+5<bm$ to have $\lambda_{q}^{4m+4+(m+1)b^3\gamma}<\lambda_{q}^{(m-2)b}$ and we chose $a$ sufficiently large to absorb the constant. Then taking expectation we obtain
	\begin{align*}
		\$R_{\textrm{oscillation}}\$_{C^0,r}\leq  \frac16  \delta_{q+2},\qquad
		\$R_{\textrm{oscillation}}\$_{L^1,1}\leq  \frac{1}{6\cdot48}  c_R\delta_{q+3}\underline{e}.
	\end{align*}
	
	\subsubsection{Estimate on the Nash error.}\label{464} Similar to transport error, we write
	\begin{align}\label{Rnash}
		w_{q+1}\cdot \nabla(v_\ell +z_\ell)=w_{q+1}^{(p)}\cdot \nabla(v_\ell +z_\ell)+w_{q+1}^{(c)}\cdot \nabla(v_\ell +z_\ell).
	\end{align}
	For the first term on the right hand side, we have for $t\in [k,k+1]$, $k\in \mathbb{Z}$
	\begin{equation*}
		w_{q+1}^{(p)}\cdot \nabla(v_\ell +z_\ell)=\sum_{j,\xi}a_{(\xi)}W_{(\xi)}\circ \Phi_{k,j}\cdot \nabla(v_{\ell}+z_{\ell}).
	\end{equation*}
	Using \eqref{vqaa}, \eqref{axiN}, \eqref{axi012}, \eqref{zq0} and chain rule, we have for $n\in \mathbb{N}_0$
	\begin{equation}\label{nashp}
		\|a_{(\xi)} \nabla(v_{\ell}+z_{\ell})\|_{C^n_{t,x}}\lesssim \ell^{-2n-2}\lambda_{q+3}^{(n+2)\gamma}.
	\end{equation}
	Combining \eqref{nashp} with \eqref{eq:Phin}, applying $\eqref{InR}$ with $n=m$ and $a$ replaced by $a_{(\xi)} \nabla(v_{\ell}+z_{\ell})$ implies 
	\begin{equation}\label{Nash1}
		\aligned
		\|\mathcal{R}(w_{q+1}^{(p)}\cdot \nabla(v_\ell +z_\ell))\|_{C^0_{t,x}}\lesssim \frac{\ell^{-2}\lambda_{q+3}^{2\gamma}}{\lambda_{q+1}^{1-\alpha}}+\frac{\ell^{-2m-4}\lambda_{q+3}^{(m+3)\gamma}+\ell^{-m-3}\lambda_{q+3}^{2\gamma}}{\lambda_{q+1}^{m-\alpha}}.
		\endaligned
	\end{equation}
	
	The second term on the right hand side of \eqref{Rnash} is
	\begin{equation*}
		w_{q+1}^{(c)}\cdot \nabla(v_\ell +z_\ell)=\sum_{j,\xi}\left( \left( \frac{\nabla a_{(\xi)}}{\lambda_{q+1}}+ i a_{(\xi)} (\nabla \Phi_{k,j}-\Id)\xi \right) \times W_{(\xi)}\circ \Phi_{k,j} \right)\cdot \nabla(v_\ell +z_\ell).
	\end{equation*}
	Similar as before, we use \eqref{vqaa}, \eqref{axiN}, \eqref{axi012} and \eqref{zq0} to have
	\begin{align*}
		\left\| \frac{\nabla a_{(\xi)}}{\lambda_{q+1}} \cdot\nabla(v_\ell +z_\ell)\right\|_{C^n_{t,x}}
		&\lesssim \frac{1}{\lambda_{q+1}} \sum_{k=0}^{n} \|\nabla a_{(\xi)}\|_{C^k_{t,x}} \|\nabla(v_\ell+z_\ell)\|_{C^{n-k}_{t,x}} 
		\\&\lesssim \frac{1}{\lambda_{q+1}}\sum_{k=0}^{n}\ell^{-2k-3} \lambda_{q+3}^{(k+2)\gamma} \ell^{-(n-k+1)}\lambda_{q+3}^{\gamma}
		\\&\lesssim \frac{\ell^{-2n-4}}{\lambda_{q+1}}\lambda_{q+3}^{(n+3)\gamma},
	\end{align*}
	and we use \eqref{nashp} and \eqref{eq:Phin} to have
	\begin{align*}
		\|a_{(\xi)} (\nabla \Phi_{k,j}-\Id)\cdot \nabla(v_\ell +z_\ell)\|_{C^0_tC^n_x} &\lesssim \sum_{l=0}^{n} \|a_{(\xi)} \nabla(v_{\ell}+z_{\ell})\|_{C^0_tC^l_x} \| \nabla\Phi_{k,j}\|_{C^0_tC_x^{n-l}}
		\\ &\lesssim \sum_{l=0}^{n}  \ell^{-2l-2}\lambda_{q+3}^{(l+2)\gamma}\ell^{-(n-l)}
		\\ &\lesssim \ell^{-2n-2}\lambda_{q+3}^{(n+2)\gamma}.
	\end{align*}
	Then combining the above two estimates with \eqref{eq:Phin} and applying $\eqref{InR}$ with $n=m$ implies 
	\begin{equation}\label{Nash2}
		\aligned
		\|\mathcal{R}(w_{q+1}^{(c)}\cdot \nabla(v_\ell +z_\ell))\|_{C^0_{t,x}} &\lesssim \frac{\ell^{-4}\lambda_{q+3}^{3\gamma}}{\lambda_{q+1}^{2-\alpha}}+\frac{\ell^{-2m-6}\lambda_{q+3}^{(m+4)\gamma}+\ell^{-m-5}\lambda_{q+3}^{3\gamma}
		}{\lambda_{q+1}^{m+1-\alpha}}
		\\&\qquad+\frac{\ell^{-2}\lambda_{q+3}^{2\gamma}}{\lambda_{q+1}^{1-\alpha}}+\frac{\ell^{-2m-4}\lambda_{q+3}^{(m+3)\gamma}+\ell^{-m-3}\lambda_{q+3}^{2\gamma}}{\lambda_{q+1}^{m-\alpha}}.
		\endaligned
	\end{equation}
	Combining \eqref{Nash1} with \eqref{Nash2} we obatin
	\begin{align*}\label{Rn} 
		\|R_{\textrm{Nash}}\|_{C^0_{t,x}} &
		\lesssim 	\|\mathcal{R}(w_{q+1}^{(p)}\cdot \nabla(v_\ell +z_\ell))\|_{C^0_{t,x}}+	\|\mathcal{R}(w_{q+1}^{(c)}\cdot \nabla(v_\ell +z_\ell))\|_{C^0_{t,x}}
		\\ &\lesssim \left( \frac{\ell^{-4}\lambda_{q+3}^{3\gamma}}{\lambda_{q+1}^{2-\alpha}} +\frac{\ell^{-2}\lambda_{q+3}^{2\gamma}}{\lambda_{q+1}^{1-\alpha}} \right) \left( 1+ \frac{\ell^{-2m-2}\lambda_{q+3}^{(m+1)\gamma}
		}{\lambda_{q+1}^{m-1}} \right)  \leq \frac{1}{6\cdot48}  c_R\delta_{q+3},
	\end{align*}
	where we used $b\geq7$ and $3b^3\gamma+3b^3\beta+b\alpha<1$ to have $\lambda_{q}^{4+2b^3\gamma+b\alpha}< \lambda_{q}^{b-2b^3\beta}$ and $\lambda_{q}^{8+3b^3\gamma+b\alpha}< \lambda_{q}^{2b-2b^3\beta}$. We also used $5m+b+5<bm$ to have $\lambda_{q}^{4m+4+(m+1)b^3\gamma
	}<\lambda_{q}^{(m-1)b}$ and chose $a$ sufficiently large to absorb the constant. Then taking expectation, we obtain
	\begin{align*}
		\$R_{\textrm{Nash}}\$_{C^0,r}\leq  \frac16  \delta_{q+2},
    \qquad
		\$R_{\textrm{Nash}}\$_{L^1,1}\leq  \frac{1}{6\cdot48}  c_R\delta_{q+3}\underline{e}.
	\end{align*}
	
	\subsubsection{Estimate on the corrector error.}\label{465}The corrector error has two pieces, one is the transport derivative of $w^{(c)}_{q+1}$ by the flow of $v_{\ell}+z_\ell$. The orther is a residual contribution from the nonlinear terms, which is easier to estimate due to \eqref{wp00} and \eqref{wc0'}. Therefore, we have
	\begin{equation*}
		\|w_{q+1}^{(c)}\mathring \otimes w_{q+1}^{(c)}\|_{C^0_{t,x}}\leq \|w_{q+1}^{(c)}\|_{C^0_{t,x}}^2\leq  \frac{\lambda_q^8\lambda_{q+3}^{4\gamma}}{\lambda_{q+1}^2}+\lambda_{q+3}^{2\gamma}\lambda_q^{-1} \leq \frac{1}{3{\cdot}6{\cdot}96}  c_R\delta_{q+3},
	\end{equation*}
	and
	\begin{equation*}
		\|w_{q+1}^{(p)}\mathring \otimes w_{q+1}^{(c)}\|_{C^0_{t,x}}\lesssim  \lambda_{q+3}^{\gamma}\left( \frac{\lambda_q^4\lambda_{q+3}^{2\gamma}}{\lambda_{q+1}}+\lambda_{q+3}^{\gamma}\lambda_q^{-\frac{1}{2}}\right) \leq \frac{1}{3{\cdot}6{\cdot}96}  c_R\delta_{q+3},
	\end{equation*}
	where we used $b\geq7$ and $4b^3\gamma+4b^3\beta<1$ to have $\lambda_{q}^{2b^3\gamma}<\lambda_{q}^{\frac12-2b^3\beta}$, $\lambda_{q}^{4+2b^3\gamma}<\lambda_{q}^{b-2b^3\beta}$ and $\lambda_{q}^{8+4b^3\gamma}<\lambda_{q}^{2b-2b^3\beta}$. Also, $a$ was chosen sufficiently large to absorb the constant. Combining the above estimates and taking expectations we obtain
	\begin{equation}\label{Cor1}
		\aligned
		\$w_{q+1}^{(c)}\mathring \otimes w_{q+1}^{(c)}+w_{q+1}^{(p)}\mathring \otimes w_{q+1}^{(c)}+w_{q+1}^{(c)}\mathring \otimes w_{q+1}^{(p)}\$_{C^0,r} &\leq  \frac{1}{12} \delta_{q+2},
		\\ \$w_{q+1}^{(c)}\mathring \otimes w_{q+1}^{(c)}+w_{q+1}^{(p)}\mathring \otimes w_{q+1}^{(c)}+w_{q+1}^{(c)}\mathring \otimes w_{q+1}^{(p)}\$_{L^1,1} &\leq  \frac{1}{6{\cdot}96}  c_R\delta_{q+3}\underline{e}.
		\endaligned
	\end{equation}

	The remaining term in the corrector error is
	\begin{align*}
		\mathcal{R} \left((\partial_t+(v_\ell+z_\ell)\cdot \nabla ) w_{q+1}^{(c)}\right).
	\end{align*}
	We use the definition of $w_{q+1}^{(c)}$ to have
	\begin{equation*}
		\aligned
		(\partial_t+(v_\ell+z_\ell)\cdot \nabla)w_{q+1}^{(c)}
		= \sum_{j,\xi}\left((\partial_t+(v_\ell+z_\ell)\cdot \nabla)\left(\frac{\nabla a_{(\xi)}}{\lambda_{q+1}}+i a_{(\xi)}(\nabla \Phi_{k,j}-\Id)\xi\right)\right)\times W_{(\xi)}\circ \Phi_{k,j}.
		\endaligned
	\end{equation*}
	Using \eqref{vqaa}, \eqref{axiN} and \eqref{axi012}, we obtain
	\begin{equation*}
		\aligned
		\left\|(\partial_t+(v_\ell+z_\ell)\cdot \nabla)\frac{\nabla a_{(\xi)}}{\lambda_{q+1}}\right\|_{C^n_{t,x}}
		& \lesssim \frac{1}{\lambda_{q+1}}\|a_{(\xi)}\|_{C^{n+2}_{t,x}} +\frac{1}{\lambda_{q+1}}\sum_{k=0}^{n}\|a_{(\xi)}\|_{C^{k+2}_{t,x}} \|v_\ell+z_\ell\|_{C^{n-k}_{t,x}}
		\\ &\lesssim \frac{1}{\lambda_{q+1}} \ell^{-2n-5}\lambda_{q+3}^{(n+3)\gamma}+ \frac{1}{\lambda_{q+1}} \sum_{k=0}^{n} \ell^{-2k-5}\lambda_{q+3}^{(k+3)\gamma}\ell^{-(n-k)}\lambda_{q+3}^\gamma
		\\ &\lesssim \frac{\ell^{-2n-5}}{\lambda_{q+1}}\lambda_{q+3}^{(n+4)\gamma}.
		\endaligned
	\end{equation*}
	Using \eqref{tranerror} and \eqref{eq:Phin}, we obtain 
	\begin{equation*}
		\aligned
		\left\|(\partial_t+(v_\ell+z_\ell)\cdot \nabla) (a_{(\xi)}(\nabla \Phi_{k,j}- \Id)) \right\|_{C^0_tC^n_x} 
		&\lesssim \sum_{l=0}^{n}\|(\partial_t+(v_\ell+z_\ell)\cdot \nabla)a_{(\xi)} \|_{C^0_tC^l_x} \|\nabla \Phi_{k,j}\|_{C^0_tC^{n-l}_x} 
		\\ &\lesssim \sum_{l=0}^{n}\ell^{-2l-3}\lambda_{q+3}^{(l+3)\gamma}\ell^{-(n-l)}
		\\ &\lesssim \ell^{-2n-3}\lambda_{q+3}^{(n+3)\gamma}.
		\endaligned
	\end{equation*}
	Combining the above two estimtaes with \eqref{eq:Phin}, applying $\eqref{InR}$ with $n=m$ implies
	\begin{equation*}\label{Rc}
		\aligned		
		\left\|\mathcal{R}\left((\partial_t+(v_\ell+z_\ell)\cdot \nabla)w_{q+1}^{(c)}\right)\right\|_{C^0_{t,x}} \lesssim& \frac{\ell^{-5}\lambda_{q+3}^{4\gamma}}{\lambda_{q+1}^{2-\alpha}}+\frac{\ell^{-2m-7}\lambda_{q+3}^{(m+5)\gamma}+\ell^{-m-6}\lambda_{q+3}^{4\gamma}}
		{\lambda_{q+1}^{m+1-\alpha}}
		\\&+\frac{\ell^{-3}\lambda_{q+3}^{3\gamma}}{\lambda_{q+1}^{1-\alpha}}+\frac{\ell^{-2m-5}\lambda_{q+3}^{(m+4)\gamma}+\ell^{-m-4}\lambda_{q+3}^{3\gamma}} {\lambda_{q+1}^{m-\alpha}} 
		\\ \lesssim& \left( \frac{\ell^{-5}\lambda_{q+3}^{4\gamma}}{\lambda_{q+1}^{2-\alpha}} +\frac{\ell^{-3}\lambda_{q+3}^{3\gamma}}{\lambda_{q+1}^{1-\alpha}} \right) \left( 1+\frac{\ell^{-2m-2}\lambda_{q+3}^{(m+1)\gamma}}
		{\lambda_{q+1}^{m-1}} \right) 
		\\ \leq&  \frac{1}{6{\cdot}96}  c_R\delta_{q+3},
		\endaligned
	\end{equation*}
	where we used $b\geq7$ and $4b^3\gamma+2b^3\beta+b\alpha<1$ to have $\lambda_{q}^{10+4b^3\gamma}<\lambda_{q}^{2b-\alpha b-2b^3\beta}$ and $\lambda_{q}^{6+3b^3\gamma}<\lambda_{q}^{b-\alpha b-2b^3\beta}$. We also used  $5m+5+b<bm$ to have  $\lambda_{q}^{4m+4+(m+1)b^3\gamma}<\lambda_{q}^{(m-1)b}$ and chose $a$ sufficiently large to absorb the constant. Taking expetation and combining \eqref{Cor1}, we obatin
	\begin{align*}
		\$R_{\textrm{corrector}}\$_{C^0,r}\leq \frac16  \delta_{q+2},
	\qquad
		\$R_{\textrm{corrector}}\$_{L^1,1}\leq \frac{1}{6{\cdot}48}  c_R\delta_{q+3}\underline{e}.
	\end{align*}
	
	\subsubsection{Estimate on the commutator error}\label{527}
	Using the definition of $z_q$ in \eqref{def z} and the pathwise inductive assumption \eqref{vqaa}, we obtain
	\begin{align}\label{Rcommutator1}
		\|R_\textrm{commutator}\|_{C^0_{t,x}}\leq 2 \|v_q+z_q\|^2_{C^0_{[t-1,t+1],x}} \leq 4(\lambda_{q+2}^{2\gamma}+\lambda_{q+3}^{\frac{\gamma}{2}}) \leq \frac{1}{10}\lambda_{q+4}^\gamma,
	\end{align}
	where we chose $a$ sufficiently large to absorb the constant in the last inequality. Using a mollification estimate and inductive assumption \eqref{vqaa} and \eqref{vqcc},  we have for any $t\in \R$ and $\delta \in (0,\frac{1}{12})$
	\begin{equation*} 
		\aligned
		&\|R_\textrm{commutator}(t)\|_{L^\infty}
		\\ &\lesssim   
		\left(\ell \|v_q\|_{C_{[t-1,t],x}^1}+\ell\|z_q\|_{C^0_{[t-1,t]}C_x^1}
		+\ell^{\frac{1}{2}-\delta}\|z_q\|_{C^{\frac{1}{2}-\delta}_{[t-1,t]} C_x^0}\right)\left(\|z_q\|_{C^0_{[t-1,t],x}}+\|v_q\|_{C^0_{[t-1,t],x}}\right).
		\endaligned
	\end{equation*}
	Taking expectation and using \eqref{estimate-zq}, \eqref{vqaa} and \eqref{vqbb}, we obtain
	\begin{equation}\label{Rcommutator}
		\aligned
		\$ R_\textrm{commutator}\$_{C^0,r} &\lesssim  \left( \ell\lambda_{q}^{\frac32}+\ell \$z_q\$_{C^1,2r}+\ell^{\frac12-\delta} \$z_q\$_{C^{\frac12-\delta}_tC^0,2r} \right)\left(  \$z_q\$_{C^0,2r}+ \$v_q\$_{C^0,2r}\right)  
		\\&\lesssim \left( \ell\lambda_{q}^{\frac32}+\ell \lambda_{q+1}^{\frac{\gamma}{8}}2rL+\ell^{\frac12-\delta}\lambda_{q+1}^{\frac{\gamma}{8}}2pL^2\right) \left( \lambda_{q}^{\beta}+ 2rL \right) 
		\\&\lesssim \lambda_{q}^{-\frac12+\beta}+\lambda_{q}^{-\frac34+\beta}\lambda_{q+1}^{\frac{\gamma}{8}}
		\leq \frac{1}{6\cdot48}c_R\delta_{q+3} \leq \frac{1}{6}\delta_{q+2},
		\endaligned
	\end{equation}
	where we used $b\gamma+4b^3\beta<1$ to have $\lambda_q^{-\frac12+\beta}<\lambda_{q}^{-2b^3\beta}$ and $\lambda_q^{-\frac34+\beta+\frac{1}{8}b\gamma}<\lambda_{q}^{-2b^3\beta}$, and $a$ was chosen sufficiently large to absorb the constant. By \eqref{Rcommutator}, we also obtain
	\begin{align}\label{com''}
		\qquad  \$ R_\textrm{commutator}\$_{L^1,1}\leq \frac{1}{6\cdot48}c_R\delta_{q+3}\underline{e}.
	\end{align}

	Summarizing over  all the estimates we have established above and by \eqref{Rcom1tx0} and \eqref{Rcommutator1}, we obtain
	\begin{equation}
		\|\mathring{R}_{q+1}\|_{C^0_{t,x}}\leq \lambda_{q+4}^\gamma, \quad 	\$\mathring{R}_{q+1}\$_{C^0,r} \leq \delta_{q+2},\quad \$\mathring{R}_{q+1}\$_{L^1,1} \leq \frac{1}{48}c_R\delta_{q+3}\underline{e}.
	\end{equation}
	Hence, \eqref{vqdd},  \eqref{vqee} and  \eqref{vqf} hold at the level $q+1$.
	\subsection{Inductive estimates for the energy}\label{s:en'}
	To conclude the proof of Proposition~\ref{p:iteration1}, we shall verify \eqref{vqg} holds at the level $q+1$. 

	\bp\label{proof:en} 
	It holds for $t\in\mR$
	\begin{align}\label{1/4'}
	\Big| e(t)(1-\delta_{q+2})-\E\|v_{q+1}(t)+{z_{q+1}}(t)\|_{L^2}^2\Big|\leq& \frac14\delta_{q+2}e(t).
	\end{align}
	
	\ep
	\begin{proof}
		By the definition of $\zeta_q$ in \eqref{zetaq}, we have
		\begin{align}\label{Eq+1'}
			\begin{aligned}
				&\left| e(t)(1-\delta_{q+2})- \E\|v_{q+1}(t)+{z_{q+1}}(t)\|_{L^2}^2 \right| \\&\leq \E \left| \int_{\mathbb{T}^3} |w_{q+1}^{(p)}|^2-3\zeta_q\dif x \right| +  \E \left|  \int_{\mathbb{T}^3}|w_{q+1}^{(c)}|^2+2w_{q+1}^{(p)}\cdot w_{q+1}^{(c)} \dif x \right| 
				\\&\qquad + \E \left| \int_{\mathbb{T}^3}|v_\ell-v_q+z_{q+1}-z_q|^2 \dif x  \right| + 2  \E \left| \int_{\mathbb{T}^3}(v_\ell-v_q+z_{q+1}-z_q)(v_q+z_q)\dif x  \right| 
				\\&\qquad +2  \E \left| \int_{\mathbb{T}^3}(v_\ell+z_{q+1})\cdot w_{q+1} \dif x \right|.
			\end{aligned}
		\end{align}
		Let us begin with the estimates for the first term  on the right hand side of \eqref{Eq+1'}. Using \eqref{-osc} and the same computation of oscillation error in \eqref{Ros}, we have for any $t\in [k,k+1]$, $k\in \mathbb{Z}$
       
		\begin{equation}\label{wp^2}
			\aligned
			&w_{q+1}^{(p)}\otimes w_{q+1}^{(p)} \\&= \sum_{j,\xi}c^{-1}_*\rho \eta_{k,j}^2\gamma_{\xi}^{(j)}\left(\Id
			- \frac{c_*\mathring{R}_\ell}{\rho} \right)^2 B_{\xi}\otimes B_{-\xi} +\sum_{j,j',\xi+\xi'\neq0} a_{(\xi)}a_{(\xi')}W_{\xi}\circ\Phi_{k,j}\otimes W_{\xi'}\circ\Phi_{k,j'}
			\\ &= \sum_{j}c^{-1}_*\rho \eta_{k,j}^2\left(\Id
			- \frac{c_*\mathring{R}_\ell}{\rho} \right)+ \sum_{j,j',\xi+\xi'\neq0} a_{(\xi)}a_{(\xi')}W_{\xi}\circ\Phi_{k,j}\otimes W_{\xi'}\circ\Phi_{k,j'}
			\\&= c^{-1}_*\rho \Id-\mathring{R}_\ell+\sum_{j,j',\xi+\xi'\neq0} a_{(\xi)}a_{(\xi')}W_{\xi}\circ\Phi_{k,j}\otimes W_{\xi'}\circ\Phi_{k,j'}.
			\endaligned
		\end{equation}
		Taking the trace on both sides of \eqref{wp^2} and using the fact that $\mathring{R}_\ell$ is traceless, we deduce for $t\in[k,k+1]$, $k\in \mathbb{Z}$
		\begin{align*}
			|w_{q+1}^{(p)}|^2-3\zeta_q=3c_*^{-1}\sqrt{\ell^2+|\mathring{R}_\ell|^2}+3(\zeta_\ell-\zeta_q)+\sum_{j,j',\xi+\xi'\neq0}\tr(a_{(\xi)}a_{(\xi')}W_{\xi}\circ\Phi_{k,j}\otimes W_{\xi'}\circ\Phi_{k,j'}).
		\end{align*}
		As a result, we have
		\begin{equation}\label{Eq1'}
			\begin{aligned}
				\E \left| \int_{\mathbb{T}^3} |w_{q+1}^{(p)}|^2-3\zeta_q\dif x \right| 
				\leq &3 c_*^{-1}(2\pi)^3\ell+3c_*^{-1}\mathbf{E}\|\mathring{R}_\ell\|_{L^1}+3\cdot(2\pi)^3|\zeta_\ell-\zeta_q|
				\\+&\E \sum_{j,j',\xi+\xi'\neq0} \left|\int_{\mathbb{T}^3}\tr(a_{(\xi)}a_{(\xi')}W_{\xi}\circ\Phi_{k,j}\otimes W_{\xi'}\circ\Phi_{k,j'})\dif x\right|.
			\end{aligned}
		\end{equation}
		In the following, we estimate each term separately. By the choice of parameters, we find
		\begin{align*}
			3 c_*^{-1}(2\pi)^3\ell=3 c_*^{-1}(2\pi)^3\lambda_{q}^{-2}\leq\frac{1}{48}\delta_{q+2}e(t),
		\end{align*}
		which requires $b^2\beta<1$ and $a$ sufficiently large to absorb the constant.
		We use \eqref{vqf} on $\mathring{R}_q$, $\supp\varphi_\ell\subset [0,\ell]$ and $c_R<c_*$  to obtain for $t\in\mR$
		\begin{align*}
			3c_*^{-1}\mathbf{E}\|\mathring{R}_\ell(t)\|_{L^1}\leq  \frac{1}{16}\delta_{q+2}e(t).
		\end{align*}
		For the third term in \eqref{Eq1'}, using \eqref{zh5/2}, \eqref{estimate-zq}, \eqref{vqbb} and \eqref{vqcc} we obtain
		\begin{align*}
			3\cdot(2\pi)^3|\zeta_\ell-\zeta_q|&\lesssim \ell \|e'\|_{C^{0}_{[t-1,t]}}+\ell  \E\|v_q\|_{C^1_{[t-1,t],x}}(\|v_q\|_{C_{[t-1,t],x}^0}+{\|z_q\|_{C_{[t-1,t],x}^0}})
			\\&\qquad \qquad +\ell^{1/2-\delta}\E\|{z_q}\|_{C_{[t-1,t]}^{1/2-\delta}C_x^0}(\|v_q\|_{C_{[t-1,t],x}^0}+\|{z_q}\|_{C_{[t-1,t],x}^0})
			\\ &\lesssim  \ell \tilde e+\ell \lambda_q^{\frac{3}{2}}(\$v_q\$_{C^0,1} + \$z_q\$_{C^0,1})+
			\ell^{\frac12-\delta}\$z_q\$_{C^{1/2-\delta}_tC^0_x,2}(\$v_q\$_{C^0,2}+\$z_q\$_{C^0,2})
			\\&\lesssim \ell \tilde e+\ell \lambda_q^{\frac{3}{2}}\$v_q\$_{C^0,1}+\ell \lambda_q^{\frac{3}{2}}\$z\$_{H^{\frac32+\sigma},1}
			\\ &\qquad \qquad +\ell^{1/2-\delta}\$z\$_{C^{1/2-\delta}_tH^{\frac32+\sigma},2}\$v_q\$_{C^0,2}+\ell^{1/2-\delta}\$z\$^2_{C^{1/2-\delta}_tH^{\frac32+\sigma},2}
			\\ &\lesssim \ell \tilde e + \ell \lambda_q^{\frac32} \lambda_q^{\beta}+\ell^{1/2-\delta}\lambda_q^{\beta}
			\leq \frac{1}{ 48}\delta_{q+2}e(t),
		\end{align*}
		where we used $4b^2\beta+2\beta<1$ to have $\lambda_q^{-\frac12+\beta}<\lambda_{q}^{-2b^2\beta}$ and $\lambda_{q}^{-\frac56+\beta}<\lambda_q^{-2b^2\beta}$, and $a$ was chosen sufficiently large to absorb the constant.
		
		Recalling the definition of $\eta_{k,j}$ and $a_{(\xi)}$ in \eqref{def axi}, we can simplify the last term in \eqref{Eq1'}. 
		For any fixed $t$ within the interval $[k,k+1]$, the sum concerning $j$ and $j'$ is  the finite sum of at most two terms. Hence, applying \eqref{R1a} with $n=m$ and using the estimates \eqref{axiN}, \eqref{axi012} and \eqref{eq:Phin} we obtain 
		\begin{align*}
			\E& \sum_{j,j',\xi+\xi'\neq0}\left| \int_{\mathbb{T}^3}\tr(a_{(\xi)}a_{(\xi')}W_{\xi}\circ\Phi_j\otimes W_{\xi'}\circ\Phi_{j'})\dif x\right| 
			\\&  \lesssim \E \sup_{j.j'} \sum_{\xi+\xi'\neq0}  \left|\int_{\mathbb{T}^3} a_{(\xi)}a_{(\xi')}  B_{(\xi)}\otimes B_{(\xi')}e ^{i\lambda_{q+1}(\xi\cdot  \Phi_{k,j} +\xi' \cdot  \Phi_{k,j'} )} \dif x \right| 
			\\ &\lesssim  \sup_{j,j'}\sum_{\xi+\xi'\neq0} \frac{\|a_{(\xi)}a_{(\xi')} \|_{C^0_tC^m_x}+ \|a_{(\xi)}a_{(\xi')} \|_{C^0_{t,x}}  (\|\nabla\Phi_{k,j}\|_{C^0_tC^m_x}+\|\nabla\Phi_{k,j'}\|_{C^0_tC^m_x})}{\lambda_{q+1}^m}
			\\& \lesssim  \frac{\ell^{-2m-1}\lambda_{q+3}^{(m+2)\gamma}}{\lambda_{q+1}^{m}}+\frac{\ell^{-m}\lambda_{q+3}^{2\gamma}}{\lambda_{q+1}^{m}}
			\\&\leq\frac{1}{48}\delta_{q+2}e(t),
		\end{align*}
		where we used $b^3\gamma+2b^2\beta<1$ and $5m+5<bm$ to have $\lambda_{q}^{4m+2+(m+2)b^3\gamma}<\lambda_{q}^{bm-2b^2\beta}$, and we chose 
		$a$ sufficiently large to absorb the constant. This completes the bound for \eqref{Eq1'}.
		
		Going back to \eqref{Eq+1'}, we have to control 
		\begin{equation}\label{E}
			\aligned
			& \E \left|  \int_{\mathbb{T}^3}|w_{q+1}^{(c)}|^2+2w_{q+1}^{(p)}\cdot w_{q+1}^{(c)} \dif x \right| 
			+\E \left| \int_{\mathbb{T}^3}|v_\ell-v_q+z_{q+1}-z_q|^2 \dif x  \right| 
			\\&\qquad+ 2 \E \left|  \int_{\mathbb{T}^3}(v_\ell-v_q+z_{q+1}-z_q)(v_q+z_q)\dif x  \right| 
			+2 \E \left| \int_{\mathbb{T}^3}(v_\ell+z_{q+1})\cdot w_{q+1} \dif x \right| .
			\endaligned
		\end{equation}
	We use the estimates \eqref{wp00} and \eqref{wc0'} to obtain
		\begin{align*}
			& \E \left|  \int_{\mathbb{T}^3}|w_{q+1}^{(c)}|^2+2w_{q+1}^{(p)}\cdot w_{q+1}^{(c)} \dif x \right|  
			\lesssim \E\|w_{q+1}^{(c)}\|_{C^0_{t,x}}^2+2\E\|w_{q+1}^{(p)}\|_{C^0_{t,x}}\|w_{q+1}^{(c)}\|_{C^0_{t,x}}
			\\ & \qquad \qquad \lesssim  \frac{\lambda_{q}^8\lambda_{q+3}^{4\gamma}}{\lambda_{q+1}^{2}}+\lambda_{q+3}^{2\gamma}\lambda_{q}^{-1}+\frac{\lambda_{q}^4\lambda_{q+3}^{3\gamma}}{\lambda_{q+1}} + \lambda_{q+3}^{2\gamma}\lambda_{q}^{-\frac12}\leq \frac{1}{48}\delta_{q+2}e(t),
		\end{align*}
		where we used $b\geq7$, $4b^3\gamma+4b^2\beta<1$ to have $\lambda_{q}^{2b^3\gamma-\frac12}<\lambda_{q}^{-2b^2\beta}$,  $\lambda_q^{4+3b^3\gamma}<\lambda_{q}^{b-2b^2\beta}$ and $\lambda_{q}^{8+4b^3\gamma}<\lambda_{q}^{2b-2b^2\beta}$, and $a$ was chosen sufficiently large to absorb the constant. Using \eqref{zh5/2}, \eqref{estimate-zq}, \eqref{vqbb}, \eqref{vq-vl1} and \eqref{z-p} we obtain
		\begin{align*}
			\E &\int_{\mathbb{T}^3}|v_\ell-v_q+z_{q+1}-z_q|^2 \dif x  
			+ 2 \E\left| \int_{\mathbb{T}^3}(v_\ell-v_q+z_{q+1}-z_q)(v_q+z_q)\dif x  \right| 
			\\&\lesssim \E\left(\|v_\ell-v_q+z_{q+1}-z_q\|_{C_{t,x}^0}^2\right)+\E\left( \|v_\ell-v_q+z_{q+1}-z_q\|_{C_{t,x}^0}\|z_q+v_q\|_{C_{t,x}^0}\right) 
			\\&\lesssim \E\left( \|v_\ell-v_q\|^2_{C_{t,x}^0}\right) +\$z_{q+1}-z_q\$^2_{C^0,2}
			\\ & \qquad+\left( \$v_\ell-v_q\$_{C^0,2}+\$z_{q+1}-z_q\$_{C^0,2}\right) (\$z_q\$_{C^0,2}+\$v_q\$_{C^0,2})
			\\&\lesssim \ell^2\lambda_{q}^3+ (\lambda_{q+1}^{-\frac{\gamma\sigma}{4}}+\lambda_{q+2}^{-2\gamma}\lambda_{q+1}^{\gamma}) +(\ell\lambda_{q}^{\frac32}+\lambda_{q+1}^{-\frac{\gamma\sigma}{8}}+\lambda_{q+2}^{-\gamma}\lambda_{q+1}^{\frac{\gamma}{4}})(\$z\$_{H^{\frac32+\sigma},2}+\lambda_q^{\beta})
			\\& \leq \lambda_{q}^{-\frac12}\lambda_{q+1}^{\beta} +\lambda_{q+1}^{-\frac{\gamma\sigma}{8}+\beta}+ \lambda_{q+2}^{-\gamma}\lambda_{q+1}^{\beta+\gamma} \leq \frac{1}{48}\delta_{q+2}e(t),
		\end{align*}
		where we used $2b\beta+\beta<\frac{\gamma\sigma}{8}$ to have $\lambda_{q}^{b\beta+2b^2\beta}<\lambda_{q}^{\frac{\gamma\sigma}{8}b}$. We also used $4b^2\beta+2b\beta<1$, $3b\beta<\gamma$ to have  $\lambda_{q}^{b\beta-\frac12}<\lambda_{q}^{-2b^2\beta}$, $\lambda_{q}^{-b^2\gamma+b\beta+b\gamma}<\lambda_{q}^{-2b^2\beta}$ and chose $a$ large enough to absorb the constant. 
		
		For the last term in \eqref{E}, we recall (cf.\eqref{wq+1'}) that $w_{q+1}$ may be written as the curl of a vector field as
		\begin{align*}
			w_{q+1}=\frac{1}{\lambda_{q+1}}\sum_{j,\xi}\mathrm{curl}(a_{\xi}B_{\xi}e^{i\lambda_{q+1}\xi\cdot\Phi_{k,j}}).
		\end{align*}
		Integrating by parts and using the estimates \eqref{zh5/2}, \eqref{estimate-zq}, \eqref{vqbb} and \eqref{axi012}, we have
		\begin{align*}
			2\E \left| \int_{\mathbb{T}^3}(v_\ell+z_{q+1})\cdot w_{q+1} \dif x \right| &\lesssim \frac{1}{\lambda_{q+1}} \E\left(  \sup_{j}\sum_{\xi}\big\|a_{\xi}B_{\xi}e^{i\lambda_{q+1}\xi\cdot\Phi_{k,j}}\big\|_{C^0_{t,x}}\|v_\ell+z_{q+1}\|_{C_{t,x}^1}\right) 
			\\ &\lesssim \frac{1}{\lambda_{q+1}} \sum_{\xi} \E (\|a_{\xi}\|_{C^0_{t,x}}\cdot\|v_\ell+z_{q+1}\|_{C^0_tC_x^1} )
			\\&\lesssim \frac{\lambda_{q+3}^\gamma}{\lambda_{q+1}}(\|v_q\|_{C^1_{[t-1,t+1],x}}+\$z_{q+1}\$_{C^1,1})
			\\&\lesssim \frac{\lambda_{q+3}^\gamma}{\lambda_{q+1}}(\lambda_{q}^{\frac32}+\lambda_{q+1}^{\frac{\gamma}{8}}\$z\$_{H^{\frac32+\sigma},1}) \leq \frac{1}{48}\delta_{q+2}e(t),
		\end{align*}
		where we used $b\geq7$, $2b^2\beta+b^3\gamma<1$ to have $\lambda_{q}^{b^3\gamma
			+\frac32}<\lambda_{q}^{b-2b^2\beta}$ and $\lambda_{q}^{b^3\gamma+b\gamma
		}<\lambda_{q}^{b-2b^2\beta}$, and $a$ was chosen sufficiently large to absorb the constant.
		Summing over the above estimates, we
		obtain 
		\begin{align*}
			\Big| e(t)(1-\delta_{q+2})-\E\|v_{q+1}(t)+{z_{q+1}}(t)\|_{L^2}^2\Big|\leq\frac14 \delta_{q+2}e(t).
		\end{align*}
		Hence, the proof of Proposition~\ref{p:iteration1} is complete.
	\end{proof}

	\section{Stationary solutions to the stochastic Euler equations}
	\label{s:4}
	
	We consider that the trajectory space defined in Section~\ref{main} is $\mathcal{T}= C(\mR;C^{\kappa})\times C(\mR;C^{\kappa})$ and we take $\kappa=\frac{\vartheta}{2}$ in this section. The corresponding shifts $S_t$, $t\in\mR$,  on trajectories are given by
	$$
	S_t(u,B)(\cdot)=(u(\cdot+t),B(\cdot+t)-B(t)),\quad t\in\mR,\quad (u,B)\in\cT.
	$$
	Recall the notion of stationary solution introduced in Definition~\ref{d:1.1}. Our main result of this section is the existence of stationary solutions which are constructed as limits of ergodic averages of solutions from Theorem~\ref{v1}. This also implies their non-uniqueness.
	\bt\label{th:s1}
	Let $u$ be a solution obtained in Theorem~\ref{v1} with $e(t)=K$ for some constant $K\geq 6{\cdot48} {\cdot(2\pi)^3}c_R^{-1}rL^2$ and satisfying  \eqref{eq:K2}, where $r,L,c_R$ are constants as in Theorem~\ref{v1}. Then there exists a sequence $T_{n}\to\infty$ and
	a stationary  solution $((\tilde\Omega,\tilde{\mathcal{F}},\tilde{\mathbf{P}}),\tilde u,\tilde B)$ to \eqref{eul1} such that
	$$
	\frac{1}{T_{n}}\int_{0}^{T_{n}}\mathcal{L}[S_{t}(u,B)] \dif t\to \mathcal{L}[\tilde u,\tilde B]
	$$
	weakly in the sense of probability measures on $\mathcal{T}$ as $n\to\infty$.   Moreover, it holds true that
	\begin{align}\label{eq:s}
		\begin{aligned}
			\tilde{\mathbf{E}}\|\tilde u\|_{L^2}^2=K. \end{aligned}\end{align}
	\et	
	\begin{proof}
		
		First, we claim that there exist $\beta',\beta''>0$ such that for any $N\in\mathbb{N}$
		\begin{align}\label{claim 6.2}
			\sup_{s\in\mR}\E\|u(\cdot+s)\|^{2r}_{C^{\beta'}([-N,N], C^{\beta''}_x)}\lesssim N.
		\end{align}
		Indeed, reviewing the proof of Theorem~\ref{v1} and using interpolation, we deduce for some $\beta'\in(0,\frac{\vartheta}{2})$ and $\beta''\in(\frac{\vartheta}{2},\vartheta)$ satisfying $\beta'+\beta''<\vartheta<\min{\{\frac{\sigma}{120{\cdot b^5}},\frac{1}{21{\cdot}b^4}\}}$
		\begin{align*}
			\sum_{q\geq0}\$v_{q+1}-v_q\$_{C^{\beta'}_tC^{\beta''}_x,2r}&\lesssim \sum_{q\geq0}\$v_{q+1}-v_q\$_{C^\vartheta_{t,x},2r} 
			\\&\lesssim \sum_{q\geq0}\$v_{q+1}-v_q\$^{1-\vartheta}_{C^0_{t,x},2r} \$v_{q+1}-v_q\$^{\vartheta}_{C^1_{t,x},2r}
			\\ &\lesssim \sum_{q\geq0} \delta_{q+1}^{\frac{1-\vartheta}{2}} \lambda_{q+1}^{\frac{3}{2}\vartheta} \delta_{q+1}^{\frac{1}{2}\vartheta}
			\leq \sqrt{3r}La^{\frac32b\vartheta}+\lambda_2^{\beta}\sum_{q\geq1} \lambda_{q+1}^{\frac{3}{2}\vartheta-\beta} <\infty.
		\end{align*}
		Hence, we conclude that $v=\lim_{q\rightarrow\infty}v_q$ exists and lies in $L^{2r}(\Omega,C^{\beta'}(\R;C^{\beta''}))$. Similarly, using \eqref{zh5/2}, \eqref{estimate-zq}, \eqref{z-p} and interpolation
		we deduce for the same $\beta'\in(0,\frac{\vartheta}{2})$ and $\beta''\in(\frac{\vartheta}{2},\vartheta)$ satisfying $\beta'+\beta''<\vartheta<\min{\{\frac{\sigma}{120{\cdot b^5}},\frac{1}{21{\cdot} b^4}\}}$ and any $p\geq1$
		\begin{align*}
			&\sum_{q\geq0}\$z_{q+1}-z_q\$_{C^{\beta'}_tC^{\beta''}_x,p}
			\\\lesssim & \sum_{q\geq0}\$z_{q+1}-z_q\$^{(\frac14-\beta')(1-\beta^{''})}_{C^0_{t,x},p}\$z_{q+1}-z_q\$^{(\frac14-\beta')\beta^{''}}_{C^0_tC^1_x,p}\$z_{q+1}-z_q\$^{\beta'(1-\beta^{''})}_{C^{\frac14}_tC^0_x,p}\$z_{q+1}-z_q\$^{\beta'\beta^{''}}_{C^\frac14_tC^1_x,p}
			\\\lesssim &\sum_{q\geq0}(\lambda_{q+1}^{-\frac{\gamma\sigma}{8}}+ \lambda_{q+2}^{-\gamma} \lambda_{q+1}^{\frac{\gamma}{8}})^{(\frac14-\beta')(1-\beta^{''})}\lambda_{q+2}^{\frac{\gamma}{8}[(\frac14-\beta')\beta''+\beta'\beta'']}
			\\\lesssim &\sum_{q\geq0} \lambda_{q+1}^{-\frac{\gamma}{8}{\cdot}\frac{\sigma}{16}}\lambda_{q+2}^{\frac{\gamma}{32}\beta''}+\lambda_{q+2}^{-\frac{\gamma}{16}+\frac{\gamma}{32}\beta''}\lambda_{q+1}^{\frac{\gamma}{8}}\lesssim \sum_{q\geq0}\lambda_{q+1}^{-\frac{\gamma}{8}{\cdot}\frac{\sigma}{16}+\frac{\gamma}{32}b\beta''}+\lambda_{q+1}^{\frac{\gamma}{8}+\frac{\gamma}{32}b\beta''-\frac{\gamma}{16}b} <\infty,
		\end{align*}
		where we used $b\geq7$, $(\frac14-\beta')(1-\beta^{''})>\frac{1}{16}$ and $\beta''<\min{\{\frac{\sigma}{4b},1\}}$ in the last two inequalities.
		Hence we obtain
		  $\lim_{q\rightarrow\infty}z_q=z$ in $L^p(\Omega,C^{\beta'}(\R;C^{\beta''}))$ for any $p\geq 1$. Then letting $u=v+z$, we obtain an $(\mathcal{F}_t)_{t\in \R}$-adapted analytically weak solution $u$ to \eqref{eul1} and 
		\begin{align*}
			\$u\$_{C^{\beta'}_tC^{\beta''}_x,2r}<\infty.
		\end{align*}
	Furthermore, by $\|u\|_{C^{\beta'}_{[-N,N]}C^{\beta''}_x}\leq \sum_{i=-N}^{N-1}\|u\|_{C^{\beta'}_{[i,i+1]}C^{\beta''}_x}$ for any $N\in\mathbb{N}$ and Minkowski's inequality, we obtain
		\begin{align*}
		\sup_{s\in\mR}\left[\E\|u(\cdot+s)\|^{2r}_{C^{\beta'}_{[-N,N]}C^{\beta''}_x}\right]^{\frac{1}{2r}} &\leq\sup_{s\in\mR}\sum_{i=-N}^{N-1}\left[ \E\|u(\cdot+s)\|^{2r}_{C^{\beta'}_{[i,i+1]}C^{\beta''}_x}\right]^{\frac{1}{2r}} 
			\\ &\leq 2N\sup_{t\in \R}\left[ \E \|u\|^{2r}_{C^{\beta'}_{[t,t+1]}C^{\beta''}_x}\right]^{\frac{1}{2r}} =2N 	\$u\$_{C^{\beta'}_tC^{\beta''}_x,2r} \lesssim N,
		\end{align*}
	where the implicit constants are independent of $N$. Hence, assertion \eqref{claim 6.2} holds.
		
		Now, we define the ergodic averages of the solution $(u,B)$ as the probability measures on the trajectory space $\cT$
		\begin{align*}
			\nu_T=\frac1T\int_0^T\mathcal{L}[S_t(u,B)]\dif t,\qquad T\geq 0.
		\end{align*}
		For $R_{N}>0$, $N\in\mathbb{N}$, by Arzelà–Ascoli Theorem and the above claim, we have that the set
		\begin{align*}
			K_M:=\cap_{N=M}^\infty\bigg\{g;\,&\|g\|_{C^{\beta'}_{[-N,N]} C^{\beta''}_x}\leq R_N\bigg\}
		\end{align*}
		is relatively compact in $C(\mR;C^{\frac{\vartheta}{2}})$ since $\frac{\vartheta}{2}<\beta''$. We then deduce that the time shifts $S_tu$, $t\in\mR$,  are tight on $ C(\mR;C^{\frac{\vartheta}{2}})$. Since  $S_tB$ is a Wiener process for every $t\in\mR$, the law of  $S_tB$ is tight. Accordingly, for any $\eps>0$ there is a compact set $ F_\eps$ in $\cT$ such that
		$$\sup_{t\in\mR}\bP(S_{t} (u,B) \in F^c_\eps)<\eps.$$
		This implies
		\begin{align*}
			\nu_T (F_\eps^c) =& \frac{1}{T} \int_0^T \mathbf{P} (S_{t} (u,B) \in F^c_\eps) \dif t
			< \eps.
		\end{align*}
	Therefore, there is a weakly converging subsequence of the probability measures $\nu_{T}$, $T\geq0$. 
		The rest of the proof follows by Jakubowski--Skorokhod representation theorem and the same argument as in \cite[Theorem 4.1]{HZZ22b}.
	\end{proof}
	Using the above result and choosing different $K$, we conclude Theorem \ref{th:main}.
	
	\appendix
	\renewcommand{\appendixname}{Appendix~\arabic{section}}
	\renewcommand{\theequation}{A.\arabic{equation}}
	\section{Beltrami waves}\label{Bw2}
	In this part, we recall the Beltrami waves from \cite[Section 5.4]{BV19} which is adapted to the convex integration scheme in Proposition~\ref{p:iteration1}.
	We point out that the construction is entirely deterministic, meaning that none of the functions below depends on $\omega$. Let us begin with the definition of Beltrami waves.

	Given $\xi\in \mathbb{S}^2\cap \mathbb{Q}^3$, let $A_\xi \in \mathbb{S}^2\cap \mathbb{Q}^3$ obey 
	\begin{align*}
		A_\xi \cdot \xi =0, \quad A_{-\xi}=A_\xi.
	\end{align*}
	We define the complex vector
	\begin{align*}
		B_\xi=\frac{1}{\sqrt{2}}(A_\xi+i\xi \times A_\xi).
	\end{align*}
	By construction, the vector $B_\xi$ has the properties
	\begin{align*}
		|B_\xi|=1,\quad B_\xi \cdot \xi =0,\quad i\xi \times B_\xi=B_\xi,\quad B_{-\xi}=\overline{B_\xi}.
	\end{align*}
	This implies that for any $\lambda \in \mathbb{Z}$, such that $\lambda \xi \in \mathbb{Z}^3$, the function
	\begin{equation}
		W_{(\xi)}:=W_{\xi,\lambda}(x):=B_\xi e^{i\lambda \xi\cdot x}
	\end{equation}
	is $\mathbb{T}^3$ periodic, divergence-free, and is an eigenfunction of the $\mathrm{curl}$ operator with eigenvalue $\lambda$. That is, $W_{(\xi)}$
	is a complex Beltrami plane wave. The following lemma states a useful property for linear combinations of
	complex Beltrami plane waves, which can be found in  \cite[Proposition 5.5]{BV19}.
	\begin{lemma}\label{Belw11}
		Let $\Lambda$  be a given finite subset of $\mathbb{S}^2\cap \mathbb{Q}^3$ such that $-\Lambda=\Lambda$, and
		let $\lambda \in \mathbb{Z}$ be such that $\lambda \Lambda \subset \mathbb{Z} $. Then for any choice of coefficients $a_\xi \in \mathbb{C}$ with $\overline{a_\xi}=a_{-\xi}$ the vector field
		\begin{equation}
			W(x)=\sum_{\xi \in \Lambda}a_\xi B_\xi e^{i\lambda \xi \cdot x} 
		\end{equation}
		is a real-valued, divergence-free Beltrami vector field $\mathrm{curl}W=\lambda W$, and thus it is a stationary solution of
		the Euler equations
		$$\div (W\otimes W)=\nabla\frac{|W|^2}{2}.$$
		Furthermore, since $B_\xi \otimes B_{-\xi}+B_{-\xi}\otimes B_\xi=2Re (B_\xi \otimes B_{-\xi}) = \mathrm{Id}- \xi \otimes \xi$, we have
		\begin{equation}\label{A3}
			\frac{1}{(2\pi)^3}\int_{\mathbb{T}^3} W\otimes W\ dx=\frac{1}{2} \sum_{\xi \in \Lambda}|a_\xi|^2(\mathrm{Id}-\xi \otimes \xi).
		\end{equation}
	\end{lemma}
	The key point of the construction is that the abundance of Beltrami waves allows us to find several such flows $v$ with the property that
	\begin{align*}
		\frac{1}{(2\pi)^3}\int_{\mathbb{T}^3} v\otimes v(t,x) \dif x
	\end{align*}
	equals a prescribed symmetric matrix $R$. Indeed we will need to select these flows so as to depend smoothly on the matrix $R$, at least when $R$ belongs to a neighborhood of the identity matrix. In view of \eqref{A3}, such selection is made possible by the following lemma from \cite[Proposition 5.6]{BV19}.
	\begin{lemma}\label{Belw22}(Geometric lemma)
		There exists a sufficiently small $c_*>0$ with the following property. Let $B_{c_*(\mathrm{Id})}$ denote the closed ball of symmetric $3\times 3$ matrices, centered at $\mathrm{Id}$, of radius $c_*$. Then, there
		exist pairwise disjoint subsets $$\Lambda_\alpha \subset \mathbb{S}^2\cap \mathbb{Q}^3, \qquad \alpha\in\{0,1\},$$
		and smooth positive functions 
		$$\gamma_\xi^{(\alpha)}\in C^\infty(B_{c_*}(\mathrm{Id})), \qquad \alpha\in\{0,1\},\  \xi \in \Lambda_\alpha,$$
		such that the following hold. For every $\xi \in \Lambda_\alpha$ we have $-\xi \in \Lambda_\alpha$ and $\gamma_\xi^{(\alpha)}=\gamma_{-\xi}^{(\alpha)}$. For each $R\in B_{c_*}(\mathrm{Id})$ we have the identity
		\begin{equation}
			R=\frac{1}{2}\sum_{\xi \in \Lambda_\alpha}\left(\gamma_\xi^{(\alpha)}(R)\right)^2(\mathrm{Id}-\xi \otimes \xi).
		\end{equation}
		We label by $n_*$ the smallest natural number such that $n_*\Lambda_\alpha \subset \mathbb{Z}^3$ for all $\alpha \in \{0,1\}$.
	\end{lemma}
	Similar as in \cite[(5.17)]{BV19}, it is sufficient to consider index sets $\Lambda_0$ and $\Lambda_1$ in Lemma~\ref{Belw22} to have 12 elements. Moreover, by abuse of notation, for $j\in \mathbb{Z}$ we denote $\Lambda_j=\Lambda_{j \ \mathrm{mod} \ 2}$. In Section~\ref{s:313}, we write $\gamma^{(\alpha)}_\xi$ as $\gamma^{(j)}_\xi$. Also, it is convenient to denote $M$ geometric constant such that
	\begin{equation}\label{A5}
		\sum_{\xi\in\Lambda_\alpha} \|\gamma^{(\alpha)}_\xi\|_{C^n(B_{c_*}(\mathrm{Id}))}\leq M
	\end{equation}
	holds for $n$ large enough, $\alpha\in\{0,1\}$ and $\xi\in\Lambda_\alpha$. This parameter is universal.

	\renewcommand{\theequation}{B.\arabic{equation}}
	\section{Proof of Proposition~\ref{51}}\label{ap:A'}  
  We concisely present the proof of Proposition~\ref{51}, following closely the calculation of \cite[Appendix B]{HZZ22b}. We only write the necessary changes compared to \cite[Appendix B]{HZZ22b}.
		By the chain rule, we have
		\begin{equation}\label{aa'}
			\|a_{(\xi)}\|_{C^{N}_{t,x}}\lesssim\sum_{m=0}^N\|\rho^{\frac12}\|_{C^m_{t,x}}
			\big\|\gamma_{\xi}\big(\Id
			-\rho^{-1}  \mathring{R}_\ell \big)\big\|_{C^{N-m}_{t,x}}.
		\end{equation}
		We estimate the two terms on the right hand side in \eqref{aa'} separately. First, using Lemma~\ref{Belw22} and \eqref{eq:rho3'} we have
		\begin{align*}
			\|\rho^{\frac12}\|_{C^0_{t,x}}\lesssim \|\mathring{R}_q\|_{C^0_{[t-1,t+1],x}}^{\frac12}+\ell^{\frac12}+\delta_{q+1}^{\frac12}\bar{e}^{\frac12}, \qquad
			\|\gamma_{\xi} (\Id
			-\rho^{-1} \mathring{R}_\ell )\|_{C^0_{t,x}}\lesssim 1.
	\end{align*}
		Second, we use \cite[(B.4)]{HZZ22b} to derive for $m\geq1$
	\begin{equation}\label{eq:rho12'}
		\aligned		
		\|\rho^{1 / 2} \|_{C^m_{t,x}}
		\lesssim  \| \rho^{1 / 2} \|_{C^0_{t,x}}+\ell^{- 1 / 2} \|\rho\|_{C^{m}_{t,x}} + \ell^{1 / 2 -m} \|\rho\|_{C^{1}_{t,x}}^m.
		\endaligned
	\end{equation}
	By the definition of $\rho$ in \eqref{eq:rho3'},
it follows from \cite[(B.1)]{HZZ22b} that for any $m\geq 1$
\begin{equation}\label{rhoN1'}
	\aligned
	\|\rho\|_{C^m_{t,x}} \lesssim \ell+\|\mathring{R}_q\|_{C^0_{[t-1,t+1],x}}+ \ell^{-m}\|\mathring{R}_q\|_{C^0_{[t-1,t+1],x}}+\ell^{-2m+1}\|\mathring{R}_q
	\|_{C^0_{[t-1,t+1],x}}^m+\ell^{-m}\delta_{q+1}\bar{e}.
	\endaligned
\end{equation}
 Hence, combining \eqref{eq:rho12'} with \eqref{rhoN1'}, we have for any $m\geq 1$
      \begin{equation}\label{rho 12}
      	\aligned		
      	\|\rho^{1 / 2} \|_{C^m_{t,x}}
      \lesssim  \ell^{\frac{1}{2}-2m}(1+\|\mathring{R}_q\|_{C^0_{[t-1,t+1],x}})^m.
      	\endaligned
      \end{equation}

   We now consider the second term on the right hand side of \eqref{aa'}. By \cite[(B.5)]{HZZ22b}, we have  for $m=0,\dots,N-1$
		\begin{equation}\label{gamma'}
			\left\|\gamma_{\xi}\left(\Id
			-\rho^{-1}\mathring{R}_\ell\right)\right\|_{C^{N-m}_{t,x}}\lesssim	 \left\| \rho^{-1} \nabla_{t,x}\mathring{R}_\ell
			\right\|_{C^0_{t,x}}^{N - m} + \left\| \rho^{-2} \mathring{R}_\ell
			\right\|_{C^0_{t,x}}^{N - m} \|\rho \|_{C^1_{t,x}}^{N - m}+\left\| \rho^{-1}\mathring{R}_\ell \right\|_{C^{N - m}_{t,x}} .
		\end{equation}
		Using $\rho\geq \ell$, $| \rho^{-1} \mathring{R}_\ell|\leq 1$, \eqref{rhoN1'} and  \cite[(B.6)]{HZZ22b}, we deduce that for any $N-m\geq 1$ 
		\begin{align}\label{esti second B1}
		\left\|\gamma_{\xi}\left(\Id
		-\rho^{-1}\mathring{R}_\ell \right)\right\|_{C^{N-m}_{t,x}}\lesssim \ell^{-2(N-m)-1}(\|\mathring{R}_q\|_{C_{[t-1,t+1],x}^0}+1)^{N-m+1}.
	\end{align}
	     Then combining \eqref{esti second B1} with the bound \eqref{rho 12} and plugging into \eqref{aa'} yields for $N\geq 2$
	     \begin{equation}\label{estimate aN}
	     	\begin{aligned}
	     		\|a_{(\xi)}\|_{C^N_{t,x}}& \lesssim \ell^{-2N-\frac{1}{2}}(1+\|\mathring{R}_q\|_{C_{[t-1,t+1],x}^0})^{N+1}
	     		\lesssim \ell^{-2N-\frac{1}{2}}\lambda_{q+3}^{(N+1)\gamma},
	     	\end{aligned}
	     \end{equation}
	     where we used inductive assumption \eqref{vqdd} for $\mathring{R}_q$ in the last inequality. 
        
Moreover, by a similar calculation as above, we also have
     	\begin{align}
     		\|a_{(\xi)}\|_{C^1_{t,x}}& \lesssim \ell^{-2}(1+\|\mathring{R}_q\|_{C^0_{[t-1,t+1],x}})^{\frac32}\leq \ell^{-2}\lambda_{q+3}^{2\gamma}, \label{estimate a1}
     		\\ 	\|a_{(\xi)}\|_{C^0_{t,x}}&\lesssim
     		\|\mathring{R}_q\|^{\frac12}_{C_{[t-1,t+1],x}^0} +\ell^{\frac12}+ \delta_{q+1}^{\frac12}\leq \lambda_{q+3}^{\gamma}. \label{estimate a0}
     	\end{align}
     	Hence, \eqref{axiN} and \eqref{axi012} follow from the estimates \eqref{estimate aN}, \eqref{estimate a1} and \eqref{estimate a0}.

	\renewcommand{\theequation}{C.\arabic{equation}}
	\section{Estimates for transport equations}	\label{s:B}
	We provide a detailed estimate of the solutions to the transport equations in Section~\ref{313'}. First, we consider the following transport equation on $[t_0,T],t_0\geq0$:
	\begin{equation}\label{eq:phi}
		\aligned
		(\partial_t+v\cdot\nabla)f&=g ,
		\\ f(t_0,x)&=f_0.
		\endaligned
	\end{equation}
	We recall the following results from \cite[Proposition D.1, (133), (134)]{BDLIS16}:
	\begin{equ}\label{eq:Gronwall}
		\|f(t)\|_{C_x^1}\leq\|f_0\|_{C_x^1}e^{(t-t_0)\|{v}\|_{C^0_{[t_0,t]}C^1_x}}+\int_{t_0}^{t}e^{(t-\tau)\|{v}\|_{C_{[t_0,t]}^0C_x^1}}\|{g(\tau)}\|_{C_x^1} \dif \tau.
	\end{equ}
	More generally, for any $N\geq2$, there exists a constant $C=C_N$ such that
	\begin{equation}\label{eq:N}
		\aligned
			\|f(t)\|_{C_x^N}\leq &\left(\|f_0\|_{C_x^N}+C(t-t_0)\|{v}\|_{C^0_{[t_0,t]}C^N_x} \|f_0\|_{C_x^1} \right) e^{C(t-t_0)\|{v}\|_{C^0_{[t_0,t]}C^1_x}}
			\\ &+\int_{t_0}^{t} e^{C(t-\tau)\|{v}\|_{C^0_{[t_0,t]}C^1_x}}\left( \|g(\tau)\|_{C_x^N}+C(t-\tau)\|{v}\|_{C^0_{[t_0,t]}C^N_x} \|g(\tau)\|_{C_x^1} \right) \dif \tau.
		\endaligned
	\end{equation}
	Consider the following case: 
	\begin{equation*}
		\aligned
		(\partial_t+v\cdot\nabla)\Phi&=0,
		\\ \Phi(t_0,x)&=x.
		\endaligned
	\end{equation*}
	We aim to estimate the $C_x^0$-norm of $\nabla\Phi(t)-\Id$. First, let $\Psi(s,x)=\Phi(s,x)-x$, then $\Psi$ satisfies the  following equation:
	\begin{equation*}
		\aligned
		(\partial_t+v\cdot\nabla)\Psi=-v,
		\\\Psi(t_0,x)=0.
		\endaligned
	\end{equation*}
	We then apply \eqref{eq:Gronwall} to conclude 
	\begin{equation}\label{d3}
		\|\nabla\Phi(t)-\Id\|_{C_x^0}=\|\Psi(t)\|_{C_x^1}\leq\int_{t_0}^{t}e^{(t-\tau)\|{v}\|_{C_{[t_0,t]}^0C_x^1}}\|{v}\|_{C^0_{[t_0,t]}C_x^1} \dif \tau=e^{(t-t_0)\Vert{v}\Vert_{C_{[t_0,t]}^0C_x^1}}-1.
	\end{equation}
	
	Now we use the above result to estimate the solution $\Phi_{k,j}$ of \eqref{eq:te}, $j\in \{0,1,...,\lceil \ell^{-1} \rceil\}$. Using \eqref{vqcc} and the definition of $z_q$ in \eqref{def z}, we obtain for any $k\in \mathbb{Z}$
	\begin{align}\label{d4}
		\ell\|{v_\ell+z_\ell\|_{C^0_{[k,k+1]}C_x^1}}\leq\ell(\lambda_q^{\frac{3}{2}}\delta_q^{\frac{1}{2}}+\lambda_{q+3}^{\frac{\gamma}{4}})\lesssim \lambda_{q}^{-\frac{1}{2}}\ll1,
	\end{align}
	where we used $b^3\gamma<1$ in the second inequality.
	Since $e^x-1\leq2x$ for $x\in[0,1]$ and combining \eqref{d3} with \eqref{d4}, we obtain 
	\begin{equation}\label{eq:a6}
		\aligned
		\sup_{t\in[k+(j-1)\ell,k+(j+1)\ell]}\|\nabla\Phi_{k,j}(t)-\Id\|_{C_x^0}&\leq e^{2\ell\|v_\ell+z_\ell\|_{C^0_{[k,k+1]}C_x^1}}-1 
		\\ &\lesssim \ell\|{v_\ell+z_\ell\|_{C^0_{[k,k+1]}C_x^1}}\lesssim \lambda_{q}^{-\frac{1}{2}}\ll1,
		\endaligned
	\end{equation}
	hence \eqref{Phija} holds. Then, \eqref{Phijb} follows from \eqref{eq:a6}, we have
	\begin{equation}\label{eq:Phi0}
		\aligned
		\sup_{t\in[k+(j-1)\ell,k+(j+1)\ell]} \| \nabla\Phi_{k,j}(t)\|_{C_x^0}\leq  \sup_{t\in[k+(j-1)\ell,k+(j+1)\ell]} \| \nabla\Phi_{k,j}(t)-\mathrm{Id}\|_{C_x^0}+1\leq 2,
		\\   \sup_{t\in[k+(j-1)\ell,k+(j+1)\ell]} \| \nabla\Phi_{k,j}(t)\|_{C_x^0}  \geq 1- \sup_{t\in[k+(j-1)\ell,k+(j+1)\ell]} \| \nabla\Phi_{k,j}(t)-\mathrm{Id}\|_{C_x^0}\geq \frac12.
		\endaligned
	\end{equation}
	
	Similarly, we use \eqref{eq:N} and \eqref{d4} again to obtain for any $n\geq 1$
	\begin{equation}\label{eq:Phin}
		\aligned
		\sup_{t\in[k+(j-1)\ell,k+(j+1)\ell]} \|\nabla \Phi_{k,j}(t)\|_{C_x^n}&\lesssim \ell \|{v_\ell+z_\ell\|_{C^0_{[k,k+1]}C_x^{n+1}}}e^{2\ell\|{v_\ell+z_\ell\|_{C^0_{[k,k+1]}C_x^1}}}
		\\&\lesssim  \ell^{-(n-1)}\|v_\ell+z_\ell\|_{C^0_{[k,k+1]}C_x^1}\lesssim \ell^{-n+\frac{1}{4}},
		\endaligned
	\end{equation}
	where we used $b^3\gamma<1$ in the last inequality. Using \eqref{vqaa}, \eqref{eq:te}, \eqref{zq0} and \eqref{eq:Phi0}, we obtain 
	\begin{equation}\label{E1}
		\aligned
		\sup_{t\in[k+(j-1)\ell,k+(j+1)\ell]}\|\partial_t\Phi_{k,j}\|_{C_x^0}\leq  \sup_{t\in[k+(j-1)\ell,k+(j+1)\ell]} \|(v_\ell+z_\ell)\cdot\nabla\Phi_{k,j}\|_{C_x^0}
		\leq \lambda_{q+3}^{\frac{\gamma}{2}}.
		\endaligned
	\end{equation}
	Differentiating both sides of \eqref{eq:te} and using \eqref{vqaa}, \eqref{vqcc}, \eqref{zq0}, \eqref{eq:Phi0} and \eqref{eq:Phin}, we also obtain

	\begin{equation}\label{E2}
		\aligned
		\sup_{t\in[k+(j-1)\ell,k+(j+1)\ell]}\|\partial_t\nabla\Phi_{k,j}\|_{C_x^0}
		 \lesssim& \sup_{t\in[k+(j-1)\ell,k+(j+1)\ell]}\|\nabla(v_\ell+z_\ell)\cdot \nabla\Phi_{k,j} \|_{C_x^0}		
		\\&+\sup_{t\in[k+(j-1)\ell,k+(j+1)\ell]}\|(v_\ell+z_\ell)\cdot \nabla^2\Phi_{k,j}\|_{C_x^0}		
		\\ \lesssim& \lambda_q^{\frac{3}{2}}\delta_q^{\frac{1}{2}}+\lambda_{q+3}^{\frac{\gamma}{4}}+(\lambda_{q+3}^{\frac{\gamma}{4}}+\lambda_{q+2}^{\gamma})\ell^{-\frac{3}{4}}\leq \lambda_q^{\frac{3}{2}}\lambda_{q+3}^{\gamma}.
		\endaligned
	\end{equation}
	The estimates \eqref{eq:Phin}--\eqref{E2} are used in Section~\ref{sss:v} and Section~\ref{sss:R} to estimate the $C^1_{t,x}$-norm of $w_{q+1}$ and the $C^0_{t,x}$-norm of $\mathring{R}_{q+1}$.

	\def\cprime{$'$} \def\ocirc#1{\ifmmode\setbox0=\hbox{$#1$}\dimen0=\ht0
		\advance\dimen0 by1pt\rlap{\hbox to\wd0{\hss\raise\dimen0
				\hbox{\hskip.2em$\scriptscriptstyle\circ$}\hss}}#1\else {\accent"17 #1}\fi}
	
\end{document}